\documentclass[12pt]{article}
\usepackage{fullpage}
\usepackage[utf8]{inputenc}
\usepackage{amsmath,amssymb,latexsym,enumerate,xfrac,hyperref, tikz, cite,microtype}
\usepackage{listings}
\usepackage{amsthm}
\usepackage[title]{appendix}

\theoremstyle{definition}
\newtheorem{theorem}{Theorem}[section]
\newtheorem{definition}[theorem]{Definition}
\newtheorem{proposition}[theorem]{Proposition}
\newtheorem{lemma}[theorem]{Lemma}
\newtheorem{corollary}[theorem]{Corollary}
\newtheorem{example}[theorem]{Example}
\DeclareMathOperator{\fix}{fix}
\newcommand{\N}{\mathbb{N}}
\newcommand{\Z}{\mathbb{Z}}
\newcommand{\iso}{\widetilde}

\newcommand{\s}{\mathcal}

\usetikzlibrary{decorations.markings}
\tikzstyle{wvertex}=[circle, draw, fill=white, inner sep=0pt, minimum size=6pt]
\tikzstyle{rvertex}=[circle, draw, fill=white, inner sep=0pt, minimum size=6pt]
\tikzstyle{bvertex}=[circle, draw, fill=lightgray, inner sep=0pt, minimum size=6pt]
\newcommand{\wvertex}{\node[wvertex]}
\newcommand{\bvertex}{\node[bvertex]}
\newcommand{\rvertex}{\node[rvertex]}

\newcommand{\threevertexgraphs}{
\wvertex (001) at (-4.5,1.9392) {1};
\wvertex (002) at (-5.1,0.9) {2};
\wvertex (003) at (-3.9,0.9) {3};

\wvertex (101) at (-1.5,1.9392) {1};
\wvertex (102) at (-2.1,0.9) {2};
\wvertex (103) at (-0.9,0.9) {3};
\path
(101) edge (102);

\wvertex (201) at (1.5,1.9392) {1};
\wvertex (202) at (0.9,0.9) {2};
\wvertex (203) at (2.1,0.9) {3};
\path
(201) edge (203);

\wvertex (301) at (4.5,1.9392) {1};
\wvertex (302) at (3.9,0.9) {2};
\wvertex (303) at (5.1,0.9) {3};
\path
(302) edge (303);

\wvertex (011) at (-4.5,-0.9) {1};
\wvertex (012) at (-5.1,-1.9392) {2};
\wvertex (013) at (-3.9,-1.9392) {3};
\path
(011) edge (012)
(011) edge (013)
(012) edge (013);

\wvertex (111) at (-1.5,-0.9) {1};
\wvertex (112) at (-2.1,-1.9392) {2};
\wvertex (113) at (-0.9,-1.9392) {3};
\path
(111) edge (113)
(112) edge (113);

\wvertex (211) at (1.5,-0.9) {1};
\wvertex (212) at (0.9,-1.9392) {2};
\wvertex (213) at (2.1,-1.9392) {3};

\path
(211) edge (212)
(212) edge (213);

\wvertex (311) at (4.5,-0.9) {1};
\wvertex (312) at (3.9,-1.9392) {2};
\wvertex (313) at (5.1,-1.9392) {3};
\path
(311) edge (312)
(311) edge (313);

}
\newcommand{\unlabeledthreevertexgraphs}{
\wvertex (001) at (-4.5,1.9392) {};
\wvertex (002) at (-5.1,0.9) {};
\wvertex (003) at (-3.9,0.9) {};

\wvertex (101) at (-1.5,1.9392) {};
\wvertex (102) at (-2.1,0.9) {};
\wvertex (103) at (-0.9,0.9) {};
\path
(101) edge (102);

\wvertex (201) at (1.5,1.9392) {};
\wvertex (202) at (0.9,0.9) {};
\wvertex (203) at (2.1,0.9) {};
\path
(201) edge (203);

\wvertex (301) at (4.5,1.9392) {};
\wvertex (302) at (3.9,0.9) {};
\wvertex (303) at (5.1,0.9) {};
\path
(302) edge (303);

\wvertex (011) at (-4.5,-0.9) {};
\wvertex (012) at (-5.1,-1.9392) {};
\wvertex (013) at (-3.9,-1.9392) {};
\path
(011) edge (012)
(011) edge (013)
(012) edge (013);

\wvertex (111) at (-1.5,-0.9) {};
\wvertex (112) at (-2.1,-1.9392) {};
\wvertex (113) at (-0.9,-1.9392) {};
\path
(111) edge (113)
(112) edge (113);

\wvertex (211) at (1.5,-0.9) {};
\wvertex (212) at (0.9,-1.9392) {};
\wvertex (213) at (2.1,-1.9392) {};

\path
(211) edge (212)
(212) edge (213);

\wvertex (311) at (4.5,-0.9) {};
\wvertex (312) at (3.9,-1.9392) {};
\wvertex (313) at (5.1,-1.9392) {};
\path
(311) edge (312)
(311) edge (313);

}

\lstnewenvironment{code}[1][]%
  {\minipage{\linewidth} 
   \lstset{basicstyle=\ttfamily\footnotesize,frame=single,#1}}
  {\endminipage}
  
\newcommand{\ainv}{\mathcal{A}^{\langle -1\rangle}}

\begin{document}
\lstset{language=Python}

\title{Combinatorial Species and Graph Enumeration}

\author{ \begin{tabular}{cc} \begin{tabular}{c}
Andy Hardt \\
Department of Mathematics \\
Carleton College \\
Northfield, MN  55057 USA \\
\texttt{hardta@carleton.edu}
\end{tabular} & \begin{tabular}{c}
Pete McNeely \\
Department of Mathematics \\
Carleton College\\
Northfield, MN  55057 USA \\
\texttt{mcneelyp@carleton.edu}
\end{tabular} \bigskip \\ \begin{tabular}{c}
Tung Phan \\
Department of Mathematics \\
Carleton College\\
Northfield, MN  55057 USA \\
\texttt{phant@carleton.edu}
\end{tabular} & \begin{tabular}{c}
Justin M. Troyka \\
Department of Mathematics \\
Carleton College\\
Northfield, MN  55057 USA \\
\texttt{troykaj@carleton.edu}
\end{tabular} \end{tabular} }

\maketitle

\begin{abstract} In enumerative combinatorics, it is often a goal to enumerate both labeled and unlabeled structures of a given type. The theory of combinatorial species is a novel toolset which provides a rigorous foundation for dealing with the distinction between labeled and unlabeled structures. The cycle index series of a species encodes the labeled and unlabeled enumerative data of that species. Moreover, by using species operations, we are able to solve for the cycle index series of one species in terms of other, known cycle indices of other species. Section 3 is an exposition of species theory and Section 4 is an enumeration of point-determining bipartite graphs using this toolset. In Section 5, we extend a result about point-determining graphs to a similar result for point-determining $\Phi$-graphs, where $\Phi$ is a class of graphs with certain properties. Finally, Appendix A is an expository on species computation using the software Sage \cite{Sage} and Appendix B uses Sage to calculate the cycle index series of point-determining bipartite graphs. \end{abstract}

\section{Introduction}

Species theory was introduced in 1981 by Andr\'{e} Joyal \cite{Joyal}. Joyal was a category theorist who realized that category theory could be applied to enumerative combinatorics. Our exposition up until Section 3.5.1 is due to Joyal's work; \cite{1998} was our source for most of this material. The rest of the exposition is due to several works on the subject, including \cite{Bousquet}, \cite{dissertation}, \cite{Andrew_Paper}, and \cite{LaBelle}. We use the computer algebra system Sage \cite{Sage} for our species calculations. Much of the Sage code used is due to Andrew Gainer-Dewar, as is most of the insight in Section 4.1.

The purpose of species theory is to make rigorous the distinction between labeled and unlabeled structures. Frequently, it is much easier to work with labeled structures than unlabeled, and the cycle index series allows us to encode both. Therefore, species theory is a valuable tool for unlabeled enumerations in particular. A species can be thought of as a machine for turning labels into structures, and this way of thinking allows us to capture the automorphism data of the structures.

Section 2 covers some graph preliminaries. The reader with a basic understanding of graph theory may skip this section and use it as a reference.

Section 3 covers the basic definition of a combinatorial species, and some species theory preliminaries. Subsection 3.1 defines combinatorial equality; 3.2 defines the associated generating series; 3.3 covers algebraic operations on species; 3.4 defines virtual species and the species logarithm $\Omega$; finally, 3.5 is an exposition of $\Gamma$-species and quotient species.

Section 4 is a demonstration of how to apply some of the species operations introduced in Section 3 to enumerate the species of point-determining bipartite graphs. This enumeration in the unlabeled case appears to be original. The computation of the associated series is done in Sage \cite{Sage} in Appendix B, after a Sage exposition in Appendx A. These calculations are a good demonstration of the power of Sage.

Section 5 is a generalization of some results by Gessel and Li \cite{2007} about point-determining graphs and graphs without endpoints. In Subsection 5.1, we generalize a standard result for point-determining graphs to point-determining versions of various classes of graphs. In 5.2, we do a similar generalization about graphs without endpoints. Finally, in 5.3, we combine these results to generalize a result relating point-determining graphs and graphs without endpoints.

\section{Graph preliminaries}
In this section we define some basic graph theoretic terms that will come up time and again in the coming sections. There is little commentary with these definitions because it is assumed that the reader is familiar with them.

\begin{definition} A \emph{graph} $G = (V,E)$ is a set $V$ together with a set $E$ of unordered pairs of distinct elements of $V$. The elements of $V$ are called \emph{vertices} of $G$, and the elements of $E$ are called \emph{edges} of $G$. \end{definition}

\begin{figure}[h!]
\[\begin{tikzpicture}

\wvertex (001) at (-4.5, 1.9) {1};
\wvertex (002) at (-4.5, 3.0) {2};
\wvertex (003) at (-3, 1.0) {5};
\wvertex(004) at (-3.5, 4){3};
\wvertex(005) at (-2, 2.0){4};
\path
(001) edge (002)
(001) edge (003)
(003) edge (005)
(002) edge (004)
(003) edge (004);

\end{tikzpicture}\]
\caption{An example of a graph.}
\label{fig:graph}
\end{figure}
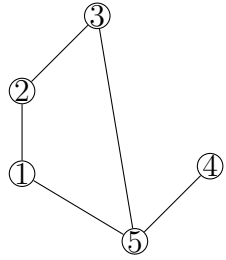

An example of a graph is seen in Figure \ref{fig:graph}. Note that by this definition, we are dealing with a simple undirected graph: a graph with no loops from a vertex to itself, no repeated edges, and no directions on its edges.

\begin{definition} Two vertices $u$ and $v$ in a graph $G$ are called \emph{adjacent} if there is an edge between them, that is if $\{u,v\}\in E$. Given a vertex $v\in V$ of a graph $G$, the \emph{neighborhood} of $v$ is the set \[\{u\in V : \{u,v\}\in E\}.\] That is, the neighborhood of $v$ is the set of vertices adjacent to it in $G$.\end{definition}

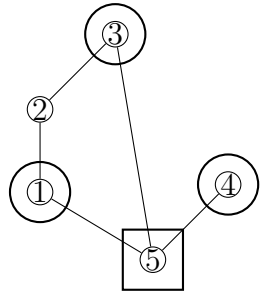
\begin{figure}[h!]
\[\begin{tikzpicture}

\draw [black, thick] (-2.0, 2.0) circle (0.4);
\draw [black, thick] (-4.5, 1.9) circle(0.4);
\draw [black, thick] (-3.4, 0.6) rectangle (-2.6, 1.4);
\draw [black, thick] (-3.5, 4) circle(0.4);

\wvertex (001) at (-4.5, 1.9) {1};
\wvertex (002) at (-4.5, 3.0) {2};
\wvertex (003) at (-3, 1.0){5};
\wvertex(004) at (-3.5, 4){3};
\wvertex(005) at (-2, 2.0){4};
\path
(001) edge (002)
(001) edge (003)
(003) edge (005)
(002) edge (004)
(003) edge (004);

\end{tikzpicture}\]
\caption{The neighborhood of $5$ (in a box) is $\{1, 3, 4\}$ (circled). }
\end{figure}

\begin{definition} The \emph{complement} of a graph $G = (V,E)$ is the graph $H = (V,E')$, where for all $u,v\in V$, $\{u,v\}\in E' \iff \{u,v\}\notin E$. That is, the complement of $G$ is the graph in which all vertices that were not adjacent in $G$ are now adjacent and all vertices that were adjacent in $G$ are not adjacent. See Figure \ref{fig:complement} \end{definition}

\begin{figure}[h!]
\[\begin{tikzpicture}

\wvertex (001) at (-4.5, 1.9) {1};
\wvertex (002) at (-4.5, 3.0) {2};
\wvertex (003) at (-3, 1.0){5};
\wvertex(004) at (-3.5, 4){3};
\wvertex(005) at (-2, 2.0){4};
\path
(001) edge (002)
(001) edge (003)
(003) edge (005)
(002) edge (004)
(003) edge (004);

\wvertex (101) at (0.5, 1.9) {1};
\wvertex (102) at (0.5, 3.0) {2};
\wvertex (103) at (2, 1.0){5};
\wvertex(104) at (1.5, 4){3};
\wvertex(105) at (3, 2.0){4};
\path
(101) edge (104)
(101) edge (105)
(102) edge (103)
(102) edge (105)
(104) edge (105);

\node[draw=none] (n1) at (-1.7,3) {};
\node[draw=none] (n2) at (-0.1,3) {};

\draw [thick] (n1) edge [<->] (n2);

\end{tikzpicture}\]
\caption{The same graph as in Figure \ref{fig:graph}, along with its complement.}
\label{fig:complement}
\end{figure}
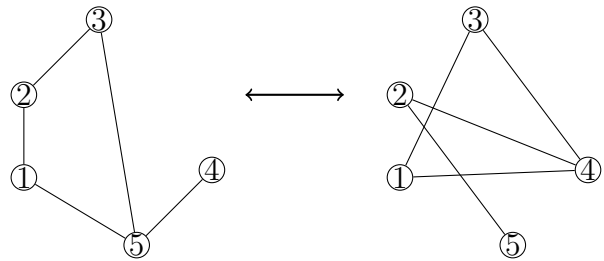

\begin{definition} A \emph{path} in a graph $G$ is a sequence $u_1,\ldots,u_n$ of vertices of $G$ so that $\{u_i,u_{i+1}\}\in E, i = 1,\ldots, n-1$. $G$ is called \emph{connected} if for any two vertices $v$ and $w$ of $G$ there is a path $u_1,\ldots,u_n$ in $G$ with $v=u_1$ and $w=u_n$. $G$ is called \emph{disconnected} if it is not connected. See Figures \ref{fig:connected} and \ref{fig:disconnected}. \end{definition}

\begin{figure}[h!]
	\begin{minipage}[b]{0.5 \linewidth}
		\[\begin{tikzpicture}

\wvertex (001) at (-3.5, 1.9) {1};
\wvertex (002) at (-3.5, 3.0) {2};
\wvertex (003) at (-2, 1.0){5};
\wvertex(004) at (-2.5, 4){3};
\wvertex(005) at (-1, 2.0){4};
\path
(001) edge (002)
(001) edge (003)
(003) edge (005)
(002) edge (004)
(003) edge (004);

\wvertex (101) at (0, 1.9) {6};
\wvertex (102) at (0, 3.0) {7};
\wvertex(104) at (1, 4){8};
\wvertex(105) at (2, 2.0){9};
\wvertex (103) at (3, 3.0){10};
\path
(101) edge (104)
(101) edge (105)
(102) edge (105)
(104) edge (105);

\path
(005) edge (101)
(005) edge (102)
(105) edge (103);

		\end{tikzpicture}\]
\caption{A connected graph.}
\label{fig:connected}
	\end{minipage}
	\begin{minipage}[b]{0.5\linewidth}
		\[\begin{tikzpicture}

\wvertex (001) at (-3.5, 1.9) {1};
\wvertex (002) at (-3.5, 3.0) {2};
\wvertex (003) at (-2, 1.0){5};
\wvertex(004) at (-2.5, 4){3};
\wvertex(005) at (-1, 2.0){4};
\path
(001) edge (002)
(001) edge (003)
(003) edge (005)
(002) edge (004)
(003) edge (004);

\wvertex (101) at (0, 1.9) {6};
\wvertex (102) at (0, 3.0) {7};
\wvertex(104) at (1, 4){8};
\wvertex(105) at (2, 2.0){9};
\wvertex (103) at (3, 3.0){10};
\path
(101) edge (104)
(101) edge (105)
(102) edge (105)
(104) edge (105);

		\end{tikzpicture}\]
\caption{A disconnected graph with 3 connected components.}
\label{fig:disconnected}
	\end{minipage}
\end{figure}

\begin{definition} A \emph{bicolored} graph is a graph where each vertex is colored with one of two colors (say, red or blue) such that no two vertices of the same color are adjacent. A \emph{bipartite} graph is a graph that admits a \emph{proper bicoloring}---a graph whose vertices can be colored with red or blue to produce a bicolored graph. See Figures \ref{fig:bipartite} and \ref{fig:bicolored}.\end{definition}

\begin{figure}[h!]
\begin{minipage}[b]{0.5\linewidth}
\[\begin{tikzpicture}

\wvertex (001) at (-2, 0) {1};
\wvertex (002) at (-2, 1) {8};
\wvertex (003) at (-2, 2){4};
\wvertex(004) at (-2, 3){3};
\wvertex(005) at (-2, 4){6};

\wvertex (101) at (2, 0.5) {9};
\wvertex (102) at (2, 1.5) {7};
\wvertex(103) at (2, 2.5){2};
\wvertex (104) at (2, 3.5){5};

\path
(001) edge (104)
(001) edge (102)
(002) edge (103)
(002) edge (101)
(002) edge (104)
(004) edge (104)
(005) edge (104);

\end{tikzpicture}\]
\caption{A bipartite graph.}
\label{fig:bipartite}
\end{minipage}
\begin{minipage}[b]{0.5\linewidth}
\[\begin{tikzpicture}

\rvertex (001) at (-2, 0) {1};
\rvertex (002) at (-2, 1) {8};
\rvertex (003) at (-2, 2){4};
\rvertex(004) at (-2, 3){3};
\rvertex(005) at (-2, 4){6};

\bvertex (101) at (2, 0.5) {9};
\bvertex (102) at (2, 1.5) {7};
\bvertex(103) at (2, 2.5){2};
\bvertex (104) at (2, 3.5){5};

\path
(001) edge (104)
(001) edge (102)
(002) edge (103)
(002) edge (101)
(002) edge (104)
(004) edge (104)
(005) edge (104);

\end{tikzpicture}\]
\caption{A bicolored graph.}
\label{fig:bicolored}
\end{minipage}
\end{figure}

A bipartite graph may be properly bicolored in multiple ways, as seen in Figure \ref{fig:anotherway}, and in the following proposition.

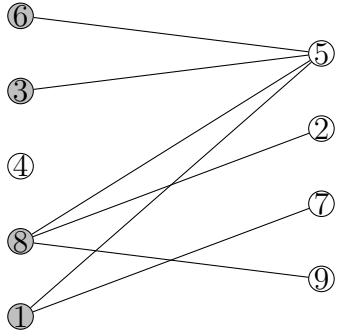
\begin{figure}[h!]
\[\begin{tikzpicture}

\bvertex (001) at (-2, 0) {1};
\bvertex (002) at (-2, 1) {8};
\rvertex (003) at (-2, 2){4};
\bvertex(004) at (-2, 3){3};
\bvertex(005) at (-2, 4){6};

\rvertex (101) at (2, 0.5) {9};
\rvertex (102) at (2, 1.5) {7};
\rvertex(103) at (2, 2.5){2};
\rvertex (104) at (2, 3.5){5};

\path
(001) edge (104)
(001) edge (102)
(002) edge (103)
(002) edge (101)
(002) edge (104)
(004) edge (104)
(005) edge (104);

\end{tikzpicture}\]
\caption{A different proper bicoloring of the bipartite graph in Figure \ref{fig:bipartite}.}
\label{fig:anotherway}
\end{figure}

\pagebreak

\begin{proposition} \label{thm:bicoloredways} If $k$ is the number of connected components of a bipartite graph, $G$, then $G$ may be properly bicolored in $2^k$ ways.\end{proposition}

\begin{proof} Let $G$ be a bipartite graph with $k$ connected components, and let $H$ be a connected component of $G$. Since $G$ is bipartite, $H$ is bipartite. So, let $A$ and $B$ partition the vertex set of $H$ so that no two vertices in $A$ are adjacent and no two vertices in $B$ are adjacent. Now without loss of generality, let $v\in A$ be colored red. Color every other vertex in $H$ so that no two vertices of the same color are adjacent (a proper bicoloring). We will show that every vertex in $A$ is colored red and every vertex in $B$ is colored blue (call this the \emph{correct coloring}).

Suppose that we do not get the correct coloring. Choose a vertex $w\neq v$ that does not have the correct coloring so that the shortest path from $v$ to $w$ is minimal. Without loss of generality, let $w\in A$, and let $u$ be the vertex adjacent to $w$ on a shortest path from $v$ to $w$. Then, $u\in B$. Thus, there is a shorter path from $v$ to $u$ than from $v$ to $w$ so, $u$ has a correct coloring. But now this means that $u$ and $w$ have the same color, a contradiction.

Thus, any connected component of $G$ can be properly bicolored in two ways: pick a vertex and color it red, or blue. Additionally, by the definition of connected, no edge in any connected component is adjacent to any vertex in another component. Therefore, we can make this choice independently for each component. Thus, there are $2^k$ ways to properly bicolor $G$.\end{proof}

In fact, a connected bipartite graph may be properly bicolored in exactly two ways: bicoloring the graph one way, or bicoloring the graph with all of the colors switched. This will be important later when we discuss quotient species.

\section{Species theory}

A species is a way of thinking about a set of combinatorial structures. Naturally speaking, a species is a function that sends a set of labels to a set of structures. Species theory allows us to manipulate such structures in ways we would not be able to otherwise.

\begin{definition} \label{def:spec} For a finite set $U$, a \emph{species} $\s{F}$ is a rule that produces a finite set $\s{F}[U]$, which is called the \emph{set of $\s{F}$-structures on $U$}. Additionally, for a bijection $\sigma: U \rightarrow V$, $\s{F}$ produces a function $\s{F}[\sigma]: \s{F}[U] \rightarrow \s{F}[V]$, which is called the \emph{transport of $\s{F}$-structures along $\sigma$}. This function $\s{F}[\sigma]$ must satisfy the following functor properties:

\begin{itemize}
\item For all bijections $\sigma: U \rightarrow V$ and $\tau: V \rightarrow W$:
$$\s{F}[\tau \circ \sigma] = \s{F}[\tau] \circ \s{F}[\sigma]$$

\item For the identity map $Id_U: U \rightarrow U$:
$$\s{F}[Id_U] = Id_{\s{F}[U]}$$
\end{itemize}
\end{definition}

For instance, the graph species $\s{G}$ is a function that turns a set of labels into the set of graphs on those labels. More rigorously,

\begin{example}
Given a set of vertices $U$, we define the species of simple graphs $\mathcal{G}[U]$ as:
$$\mathcal{G}[U] = \left\{(U, E) \left| E \subseteq \binom{U}{2} \right.\right\}$$
where $\binom{U}{2}$ is the set of unordered pairs of distinct elements of $U$. Moreover, for any bijection $\sigma: U \rightarrow V$, we define the transport of $\mathcal{G}$-structures along $\sigma$ as the relabeling of vertices in $U$ by the vertices in $V$ using the bijection $\sigma$. 

Given $\{1,2,3\}$ as the set of labels, we obtain the set of all possible graphs on three labels, $\s{G}[\{1,2,3\}]$, as depicted in Figure \ref{3labeled}.

	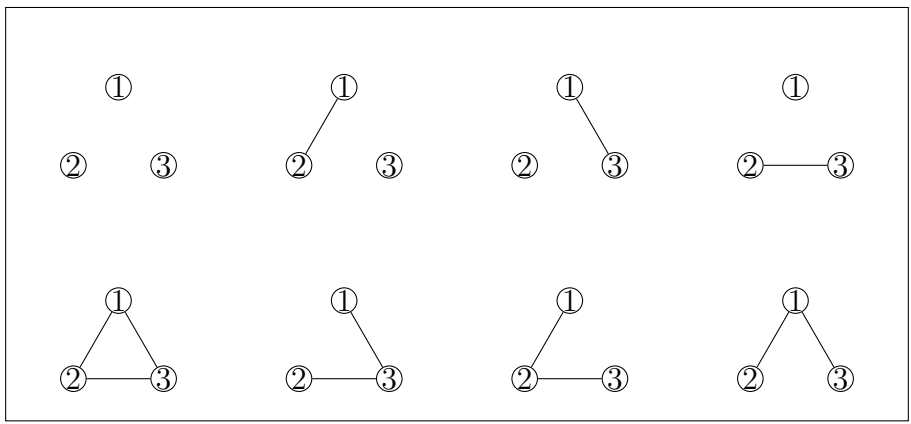
\begin{figure}[h!]
			\[\begin{tikzpicture}
            \threevertexgraphs
            \draw (-6, -2.5) rectangle (6,3);

			\end{tikzpicture}\]
            \caption{The eight labeled graphs on three vertices.}\label{3labeled}
	\end{figure}
 
Using species transports, we can permute the labels of such graphs. Consider the permutation $\sigma = (123)$. The transport $G[\sigma]$ maps the edgeless graph to itself and the complete graph to itself. It also creates an orbit of length three among those graphs with exactly one edge, and an orbit of length three among those graphs with exactly two edges (depicted in Figure \ref{123orb}).

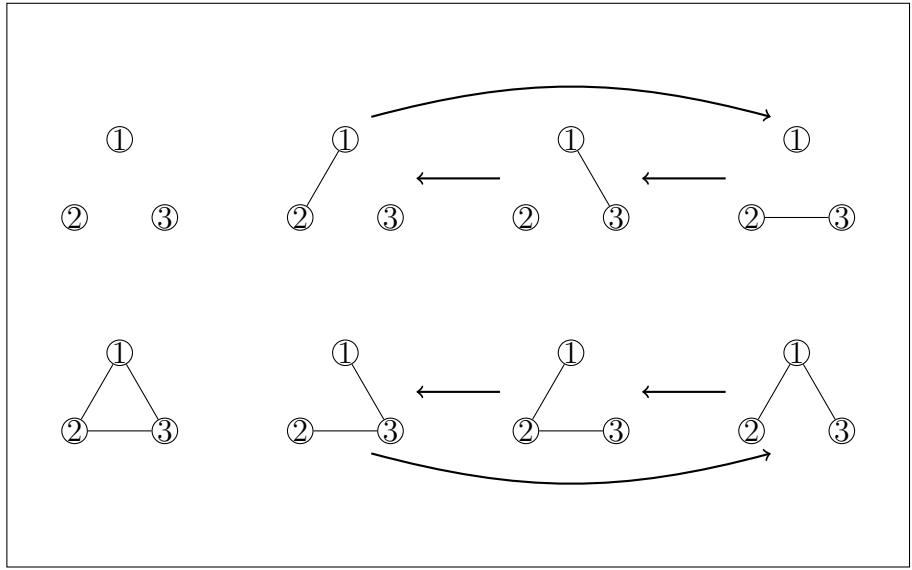
\begin{figure}[h!]
\begin{minipage}[b]{0.15\linewidth}
\[\begin{tikzpicture}
\node[draw=none] at (0:0) {};
\end{tikzpicture}\]
\end{minipage}
\begin{minipage}[b]{0.85\linewidth}
\begin{tikzpicture}

\threevertexgraphs

\node[draw=none] (h6) at (0.7,1.42) {};
\node[draw=none] (h7) at (2.3,1.42) {};

\node[draw=none] (h6') at (0.7,-1.42) {};
\node[draw=none] (h7') at (2.3,-1.42) {};
\node[draw=none] (h8') at (3.7,-1.42) {};

\node[draw=none] (h1) at (-1.3,2.2) {};
\node[draw=none] (h5) at (-0.7,1.42) {};

\node[draw=none] (h2) at (4.3,2.2) {};
\node[draw=none] (h8) at (3.7,1.42) {};

\node[draw=none] (h3) at (-1.3,-2.2) {};
\node[draw=none] (h5') at (-0.7,-1.42) {};

\node[draw=none] (h4) at (4.3,-2.2) {};

\put(-12,47) {$\displaystyle{(123)}$}
\put(-12,-33) {$\displaystyle{(123)}$}
\put(73,47) {$\displaystyle{(123)}$}
\put(73,-33) {$\displaystyle{(123)}$}
\put(30,80) {$\displaystyle{(123)}$}
\put(30,-85){$\displaystyle{(123)}$}

\draw [thick] (h2) edge[out=165, in=15, <-] (h1);
\draw [thick] (h4) edge[out=195, in=-15, <-] (h3);
\draw [thick] (h5) edge [<-] (h6);
\draw [thick] (h7) edge [<-] (h8);
\draw [thick] (h5') edge [<-] (h6');
\draw [thick] (h7') edge [<-] (h8');

\draw (-6, -3.75) rectangle (6,3.75);

\end{tikzpicture}
\end{minipage}
\caption{Applying $\s{G}[(123)]$ to the eight labeled graphs on three vertices.}\label{123orb}
\end{figure}

\end{example}

\subsection{Associated series}
For a species of structure $\s{F}$, there exist three power series that allow us to enumerate $\s{F}$-structures. These power series are the exponential generating series, the type generating series, and the cycle index series.

\subsubsection{Exponential generating series}
The exponential generating series is a powerful tool that allows us to enumerate the labeled $\s{F}$-structures in $\s{F}[U]$ for a species $\s{F}$. Such information is recorded in the coefficients of the infinite series.

\begin{definition}
The \emph{exponential generating series} of a species of structures $\s{F}$ is defined as:
$$ \s{F}(x) = \sum_{n=0}^\infty f_n \frac{x^n}{n!}$$
where $f_n = |\s{F}[n]| = \text{the number of $\s{F}$-structures on a set of n elements}$.
\end{definition}

\begin{example} There are $f_n = n!$ linear orderings on a set of size $n$. Thus, the exponential generating function for the species $\s{L}$ of linear orderings is $\s{L}(x) = \sum_{n=0}^\infty \frac{n!}{n!}x^n =\sum_{n=0}^\infty x^n = \frac{1}{1-x}$.
\end{example}

\subsubsection{Type generating series}
The type generating series allows us to enumerate unlabeled $\s{F}$-structures. This can be seen as computing the number of different ``shapes" that an $\s{F}$-structure of a given size can have, irrespective of labels. In other words, this can be thought of as isomorphism classes of labeled structures under permuting the labels. For instance, whereas there are eight labeled $\s{G}$-structures on three vertices (Figure \ref{3labeled}), there are four unlabeled $\s{G}$-structures---four ``shapes"---on three vertices (Figure \ref{unlabeled}).

	\begin{figure}[h!]
			\[\begin{tikzpicture}
\wvertex (001) at (-4.5,1.9392) {};
\wvertex (002) at (-5.1,0.9) {};
\wvertex (003) at (-3.9,0.9) {};

\wvertex (201) at (1.5,1.9392) {};
\wvertex (202) at (0.9,0.9) {};
\wvertex (203) at (2.1,0.9) {};
\node[draw=none] (h6) at (0.7,1.42) {};
\node[draw=none] (h7) at (2.3,1.42) {};
\path
(201) edge (203);

\wvertex (011) at (-4.5,-0.9) {};
\wvertex (012) at (-5.1,-1.9392) {};
\wvertex (013) at (-3.9,-1.9392) {};
\path
(011) edge (012)
(011) edge (013)
(012) edge (013);

\wvertex (211) at (1.5,-0.9) {};
\wvertex (212) at (0.9,-1.9392) {};
\wvertex (213) at (2.1,-1.9392) {};
\node[draw=none] (h6') at (0.7,-1.42) {};
\node[draw=none] (h7') at (2.3,-1.42) {};
\node[draw=none] (h8') at (3.7,-1.42) {};
\path
(211) edge (212)
(212) edge (213);

\unlabeledthreevertexgraphs

\node[draw=none] (h6) at (0.7,1.42) {};
\node[draw=none] (h7) at (2.3,1.42) {};

\node[draw=none] (h6') at (0.7,-1.42) {};
\node[draw=none] (h7') at (2.3,-1.42) {};
\node[draw=none] (h8') at (3.7,-1.42) {};

\node[draw=none] (h1) at (-1.3,2.2) {};
\node[draw=none] (h5) at (-0.7,1.42) {};

\node[draw=none] (h2) at (4.3,2.2) {};
\node[draw=none] (h8) at (3.7,1.42) {};

\node[draw=none] (h3) at (-1.3,-2.2) {};
\node[draw=none] (h5') at (-0.7,-1.42) {};

\node[draw=none] (h4) at (4.3,-2.2) {};

\draw [rounded corners] (-5.5,2.3) rectangle (-3.5,0.5);
\draw [rounded corners] (-2.5, 2.3) rectangle (5.5, 0.5);
\draw [rounded corners] (-5.5, -0.5) rectangle (-3.5, -2.3);
\draw [rounded corners] (-2.5, -0.5) rectangle (5.5, -2.3);

			\end{tikzpicture}\]
            \caption{The four isomorphism classes of the unlabeled graphs on three vertices.} \label{unlabeled}
	\end{figure}
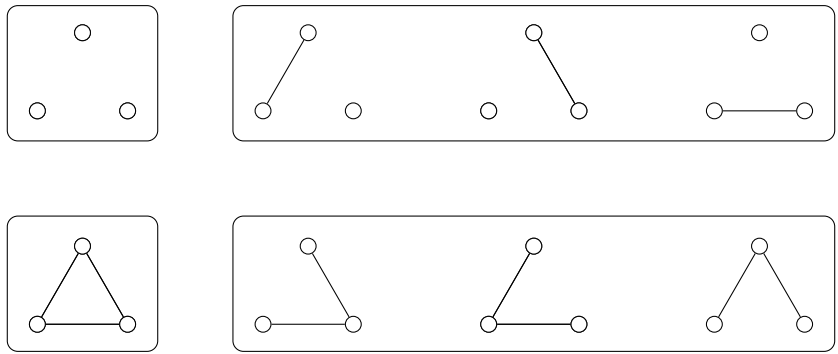

\begin{definition}\label{equiv}
For a set $U$ of the form $U=\{1, 2 \cdots n\} = [n]$, we define the equivalence relation $\sim$ for any pair $s, t \in \s{F}[n]$: 

$$s \sim t \text{  if and only if s and t have the same isomorphism type.}$$
\end{definition}

By Definition \ref{equiv}, an isomorphism type of $\s{F}$-structures of order $n$ is defined as an equivalence class of $\s{F}$-structures on $n$. With this in mind, we define the type generating series:

\begin{definition}
Let $T(\s{F}_n)$ be the quotient set, denoted $\s{F}[n]/\sim$, of isomorphism classes of $\s{F}$-structures of order $n$. The \emph{type generating series} of a species of structures $\s{F}$ is the formal power series:
$$\iso{\s{F}}(x) = \sum_{n=0}^{\infty}\iso{f_n} x^n $$
where $\iso{f_n} = |T(\s{F}_n)|$ is the number of unlabeled $\s{F}$-structures of order $n$.
\end{definition}

\begin{example}Given a set of $n$ elements and the species of permutations, $\s{S}$, we compute the value $|T(\s{S}_n)|$. For $s, t \in \s{S}$ to be equivalent, the lengths of disjoint cycles when decomposing $s$ must match the lengths of the cycles of $t$. So $\iso{f_n}$ is the number of partitions of the number $n$. Then, we can verify that $\iso{\s{S}}(x) = \prod_{k=0}^\infty \frac{1}{1-x^k}$.\end{example}

\subsubsection{Cycle index series}
By permuting the labels and checking which structures are fixed by each permutation, we can get information about the way they relate to each other. In particular, a Burnside's Lemma computation allows us to enumerate unlabeled structures.

The cycle index series is defined as follows.

\begin{definition}
The \emph{cycle index series} of a species of structures $\s{F}$ is the formal power series: 
$$Z_{\s{F}}(p_1, p_2, p_3 \cdots) = \sum_{n=0}^{\infty} \frac{1}{n!} \sum_{\sigma \in S_n} (\fix \s{F}[\sigma]) p_\sigma$$

where $S_n$ denotes the permutation group of $[n]$, $\fix \s{F}[\sigma] = (\s{F}[\sigma])_1$ is the number of $\s{F}$-structures on $[n]$ fixed by $\s{F}[\sigma]$, $p_\sigma$ is the monomial term $p_1^{\sigma_1} p_2^{\sigma_2} p_3^{\sigma_3} \cdots p_n^{\sigma_n}$, and $\sigma_i$ is the number of $i$-cycles of $\sigma$.
\end{definition}

\begin{example}
Consider the species of graphs, $\s{G}$. We can compute how much the graphs on the first few values of $n$ contribute to the overall cycle index series of $\s{G}$.
\begin{itemize}
\item For $n=2$, $\s{G}[\{1,2\}]$ contains only two graphs: the complete graph on two vertices ($K_2$) and its complement ($K_2^c$). So, applying any permutation of $S_2$ to these graphs leaves us with the original graphs. Thus we have:
$$\sum_{\sigma \in S_2} (\fix \s{G}[\sigma]) p_\sigma = 2p_1^2 + 2p_2$$
\item For $n=3$, $\s{G}[\{1,2,3\}]$ contains eight graphs as shown above. 

\begin{itemize}
\item Consider the permutation $\sigma = (123)$. In Figure \ref{123orb}, we see that $\sigma$ fixes two out of the eight structures of $\s{G}[\{1,2,3\}]$. Thus, its term in the cycle index series is $2p_3$.
\item Consider the permutation $\sigma = (12)$. One can see that the permutation fixes four structures while permuting two pairs of structres in $\s{G}[\{1,2,3\}]$. Thus, its term in cycle index series is $4p_1 p_2$. 
\end{itemize}
\end{itemize}
\end{example}

\subsubsection{The cycle index series as a generalization}

The cycle index series of a species is simultaneously a generalization of both the generating series and the type generating series. We can indeed obtain both of these series from the cycle index series by the following theorem, a proof of which is in \cite[Theorem 1.2.8]{1998}.

\begin{theorem}
For any species of structures $\s{F}$:
$$\s{F}(x) = Z_{\s{F}}(x, 0, 0, \cdots)$$
$$\iso{\s{F}}(x) = Z_{\s{F}}(x, x^2, x^3, \cdots) $$
\end{theorem}

The computation of the type generating series from the cycle index series is an application of Burnside's Lemma. For many species, a direct computation of the type generating series is challenging because it is sometimes hard to deal with the structural symmetries of the objects. In some of these cases, we can still enumerate the species by computing the cycle index using the species operations described in the next section. But first, we introduce some examples of species.

\subsubsection{Examples of Basic Species}
\begin{itemize}
\item The \emph{empty species} ($\mathbf{0}$) is defined as the species that puts zero structures on a set $U$. That is, $\mathbf{0}[U]= \emptyset$ for every set $U$.

\item The \emph{empty set species} ($\mathbf{1}$) is defined as the species that puts one structure on the empty set, and puts no structures on anything else. That is, for a set $U$, 
$$\mathbf{1}[U] = \left\{
\begin{array}{lr}
\{U\} & \text{if } U = \emptyset, \\
\emptyset & \text{otherwise.}
     \end{array}
   \right.$$

\item The \emph{singleton species} ($X$) is defined as the species that puts a structure on the set $U$ if $|U| = 1$ and no structure on $U$ otherwise. That is, for a set $U$,
$$X[U] = \left\{
\begin{array}{lr}
\{U\} & \text{if } |U| = 1, \\
\emptyset & \text{otherwise.}
     \end{array}
   \right.$$

\item The \emph{permutation species} ($\s{S}$) is defined as the species whose structures on $U$ are the permutations of $U$. The transports of the permutation species work in the following way: given a transport $\sigma$ and a permutation $\pi$, $\s{S}[\sigma](\pi) = \sigma\pi\sigma^{-1}$.

\item The  \emph{linear ordering species} ($\s{L}$) is defined as the species whose structures on $U$ are the linear orderings on $U$. The transports of the linear orderings species work in the following way: given a transport $\sigma$ and a linear ordering $\{a_1, a_2, \cdots, a_k\}$, $$\s{L}[\sigma](\{a_1, a_2, \cdots, a_k\}) = \{\sigma(a_1), \sigma(a_2), \cdots, \sigma(a_k)\}.$$

\item The \emph{set species} ($\s{E}$) is defined as the species whose only structure on $U$ is the set $U$ itself. That is, $\s{E}[U] = \{U\}$.
\end{itemize}
\begin{tabular}{ccccc}
Species & Symbol & Exp. Gen. Series & Type Gen. Series & Cycle Index Series\\ 
\hline
Empty & $\mathbf{0}$ & 0 & 0 & 0\\
Empty Set & $\mathbf{1}$ & 1 & 1 & 1\\
Singleton & $X$ & $x$ & $x$ & $p_1$\\
Permutations & $\s{S}$ & $\frac{1}{1-x}$ & $\prod_{k=1}^\infty \frac{1}{1-x^k}$ & $\prod_{k=1}^\infty \frac{1}{1-p_k}$\\
Linear Orderings & $\s{L}$ & $\frac{1}{1-x}$ & $\frac{1}{1-x}$ & $\frac{1}{1-p_1}$\\
Set & $\s{E}$ & $e^x$ & $\frac{1}{1-x}$ & $\exp(p_1 + \frac{p_2}{2} + \frac{p_3}{3} + \cdots$)\\

\end{tabular}

\subsection{Combinatorial equality} \label{sec:equality}
Two species of structures $\s{F}$ and $\s{G}$ are said to be combinatorially equal when they are isomorphic, which is defined as follows:

\begin{definition}
A \emph{species isomorphism} from $\s{F}$ to $\s{G}$ is a family of bijections $\alpha_U: \s{F}[U]\rightarrow \s{G}[U]$ which satisifies the naturality condition: For any bijection $\sigma: U \rightarrow V$ and any $\s{F}$-structure $s \in \s{F}[U]$, we have $\s{G}[\sigma] (\alpha_U(s)) = \alpha_V(\s{F}[\sigma](s))$.
\end{definition}

%
%
%
%
For species $\s{F}$ and $\s{G}$, we denote their equality as $\s{F} = \s{G}$. If two species are combinatorially equal, it does not necessarily mean that the species' structure sets are exactly identical. All it means is that their combinatorial properties are the same, including their three associated series. In short, if $\s{F} = \s{G}$, then their structure sets are the same size, and their transports act in structurally the same way. On the other hand, it is not enough to have simply the same number of structures, as we see in the following example.

\begin{example}
Consider the species of linear orderings ($\s{L}$) and the species of permutations ($\s{S}$). Despite having the same exponential generating series, their transports differ. Moreover, their type generating series and cycle index series are not equal. Therefore, $\s{L} \neq \s{S}$.
\end{example}

\subsection{Algebraic operations on species}
To find the cycle index series of a species, it is not always feasible to compute the sum directly. Instead, we may build a species out of other species by defining operations such as addition and multiplication on species. We now make a sequence of precise technical definitions, each followed by explanation and examples that illuminate its underlying meaning. This material is covered by \cite{1998}.

\subsubsection{Addition}
\begin{definition} Let $\s{F}$ and $\s{G}$ be species. Then their \emph{sum} $\s{F}+\s{G}$ is the species where
\[ (\s{F}+\s{G})[U] = \s{F}[U] \sqcup \s{G}[U], \]
where $A \sqcup B$, the disjoint union of $A$ and $B$, is the set of $A$-elements colored red and $B$-elements colored blue, and
\[ (\s{F}+\s{G})[\sigma](s) = \left\{ \begin{array}{ll}
F[\sigma](s) & \mbox{if } s \in \s{F}[U] \\
G[\sigma](s) & \mbox{if } s \in \s{G}[U].
\end{array} \right. \]
\end{definition}
This means that an $\s{F}+\s{G}$-structure is a red $\s{F}$-structure or a blue $\s{G}$-structure. Two red (resp.~blue) $\s{F}+\s{G}$-structures are isomorphic if and only if they are isomorphic as $\s{F}$-structures (resp.~$\s{G}$-structures). It is easily seen that addition of species is associative and commutative.

Species addition behaves nicely with generating series, as seen in this proposition, where the proof can be found in \cite{1998}.
\begin{proposition}
For any species $\s{F}$ and $\s{G}$,
\begin{itemize}
\item $(\s{F} + \s{G})(x) = \s{F}(x) + \s{G}(x)$;
\item $\iso{(\s{F} + \s{G})}(x) = \iso{\s{F}}(x) + \iso{\s{G}}(x)$;
\item $Z_{\s{F}+\s{G}}(p_1, p_2, \ldots) = Z_{\s{F}}(p_1, p_2, \ldots) + Z_{\s{G}}(p_1, p_2, \ldots)$.
\end{itemize}
\end{proposition}

\begin{example} \hspace{1 pc}
\begin{itemize}
\item If $\s{A}$ is the species of trees, $\s{B}$ is the species of forests, and $\s{B}^*$ is the species of disconnected forests, then $\s{B} = \s{A} + \s{B}^*$.

\item Consider the species $\s{G}^C$ (connected graphs) and $\s{B}$ (forests). Then $\s{G}^C + \s{B}$ is the species of structures that are either connected graphs or forests. Since a tree is a connected forest, $\s{G}^C + \s{B}$ includes every tree twice: once in $\s{G}^C$ and once in $\s{B}$. Also, even though they may be identical in structure, none of the $\s{G}^C$-trees are considered isomorphic with the $\s{B}$-trees because they come from different terms in the sum.

\item A simpler example is $\mathbf{1}+X$. Since $\mathbf{1}$ is the species of the empty set and $X$ is the species of singleton sets, $\mathbf{1}+X$ is the species of sets of size at most 1.
\end{itemize}

\end{example}

\subsubsection{Multiplication}
\begin{definition} Let $\s{F}$ and $\s{G}$ be species. Then their \emph{product} $\s{F} \cdot \s{G}$ is the species such that:
\begin{align*} (\s{F} \cdot \s{G})[U] &= \bigcup_{S \subseteq U} \s{F}[S] \times \s{G}[U \setminus S] \\
&= \{(s, t) : \mbox{ there is } S \subseteq U \mbox{ such that } s \in \s{F}[S] \mbox{ and } t \in \s{G}[U \setminus S] \},
\end{align*}
and
\[ (\s{F} \cdot \s{G})[\sigma](s, t) = (\s{F}[\sigma_1](s), \s{G}[\sigma_2](t)), \]
where $\sigma_1$ and $\sigma_2$ are the restrictions of $\sigma$ to the label sets of $s$ and $t$ respectively.
\end{definition}
That is, we form an $(\s{F} \cdot \s{G})$-structure by partitioning the label set into two parts, putting an $\s{F}$-structure on the first part, and putting a $\s{G}$-structure on the second part. Two $(\s{F} \cdot \s{G})$-structures are isomorphic if and only if their $\s{F}$-structure parts are isomorphic and their $\s{G}$-structure parts are isomorphic. In short, an $(\s{F} \cdot \s{G})$-structure is an ordered pair of an $\s{F}$-structure and a $\s{G}$-structure (see Figure \ref{fig:multEx}). It is easily checked that species multiplication is associative and commutative, and that it distributes with species addition.

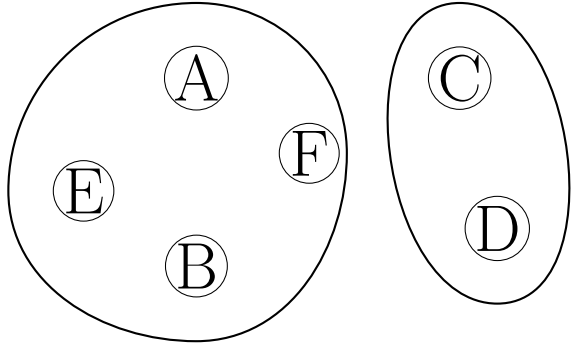
\begin{figure}[h!]
\[\begin{tikzpicture}

\wvertex (0) at (0, 0) {\huge{A}};
\wvertex (1) at (0, -2.5) {\huge{B}};
\wvertex (2) at (3.5, 0) {\huge{C}};
\wvertex (3) at (4, -2) {\huge{D}};
\wvertex (4) at (-1.5, -1.5) {\huge{E}};
\wvertex (5) at (1.5, -1) {\huge{F}};
\node[draw=none] (h) at (0, 2) {};

\put (0, 35) {{\huge{$\mathcal{F}_1$}}};
\put (80, 35) {{\huge{$\mathcal{F}_2$}}};

\draw [thick] (0, -3.5) to [out=180, in=270] (-2.5, -1.5)
to [out=90, in=180] (0,1) to [out=0, in=90] (2, -1)
to [out=270, in=0] (0, -3.5);
\draw [thick] (3.5, 1) to [out=0, in=0] (4, -3) to [out=180, in=180] (3.5, 1);

\end{tikzpicture}\]
\caption{This is a diagram of an arbitrary structure of the species $\s{F}_1\cdot \s{F}_2$. Given the label set $U = \{A,B,C,D,E,F\}$, we partition $U$ into two sets, in this case $\{A,B,E,F\}$ and $\{C,D\}$, and put an $\s{F}_1$ structure on the first set and an $\s{F}_2$ structure on the second set.}
\label{fig:multEx}
\end{figure}

Species multiplication also behaves nicely with generating series, as seen in this proposition. The proof can be found in \cite{1998}.
\begin{proposition}
For any species $\s{F}$ and $\s{G}$,
\begin{itemize}
\item $(\s{F} \cdot \s{G})(x) = \s{F}(x) \, \s{G}(x)$;
\item $\iso{(\s{F} \cdot \s{G})}(x) = \iso{\s{F}}(x) \, \iso{\s{G}}(x)$;
\item $Z_{\s{F} \cdot \s{G}}(p_1, p_2, \ldots) = Z_{\s{F}}(p_1, p_2, \ldots) \, Z_{\s{G}}(p_1, p_2, \ldots)$.
\end{itemize}
\end{proposition}

\begin{example} \hspace{1 pc}
\begin{itemize}
\item If $\s{E}$ is the set species, then an ($\s{E} \cdot \s{E}$)-structure is just a partition of the label set into two distinguishable parts. Since ``in the second part'' is equivalent to ``not in the first part'', the species $\s{E} \cdot \s{E}$ is equal to $\textsc{Sub}$, the species of subsets of the label set. Then
\[ \textsc{Sub}(x) = (\s{E} \cdot \s{E})(x) = \s{E}(x) \s{E}(x) = (e^x)(e^x) = e^{2x} = \sum_{n=0}^{\infty} \frac{2^n}{n!}. \]
This shows that a set of size $n$ has $2^n$ subsets, without actually having to count them.

\item Let $X$ be the species of singleton sets, $\s{L}$ the species of linear orderings, and $\s{L}_+$ the species of non-empty linear orderings. We can write
\[ \s{L}_+ = X \cdot \s{L}, \]
because every non-empty linear ordering can be partitioned into a first element (a singleton) followed by the rest (a linear ordering). Since every linear ordering is empty or non-empty, $\s{L} = 1 + \s{L}_+$. Therefore,
\[ \s{L} = 1 + X \cdot \s{L}. \]
This relation also holds for the respective cycle-index series:
\[ Z_{\s{L}}(p_1, p_2, \ldots) = 1 + p_1 Z_{\s{L}}(p_1, p_2, \ldots). \]
This allows us to solve for $Z_{\s{L}}(p_1, p_2, \ldots)$, proving that $Z_{\s{L}}(p_1, p_2, \ldots) = \frac{1}{1-p_1}$.

\item Recall that $\mathbf{0}$ is the empty species (which has no structures) and $\mathbf{1}$ is the empty-set species (whose only structure is the empty set). Notice that, for any species $\s{F}$, we have $\s{F} \cdot \mathbf{0} = \mathbf{0}$, because there are no $\mathbf{0}$-structures. We also have $\s{F} \cdot \mathbf{1} = \s{F}$, because the only way to partition the label set between an $\s{F}$-structure and the empty set is to give all of them to the $\s{F}$-structure and none of them to the empty set.
\end{itemize}

\end{example}

\subsubsection{Composition}
\begin{definition} Let $\s{F}$ and $\s{G}$ be species such that $\s{G}[\emptyset] = \emptyset$. We will define a species $\s{F} \circ \s{G}$, called the \emph{composition} (or \emph{substitution}) of $\s{F}$ and $\s{G}$. Its structure set $(\s{F} \circ \s{G})[U]$ is given by:
\[ \{ (s, T) : \mbox{$s \in \s{F}[T]$ and $T$ is a set of $\s{G}$-structures whose label sets partition $U$} \}. \]
Let $(s, T) \in (\s{F} \circ \s{G})[U]$. For each $t \in T$, let $\sigma_t$ be the restriction of $\sigma$ to the label set of $t$. Let $\tau$ be a function from $T$ defined by
\[ \tau(t) = \s{G}[\sigma_t](t). \]
Then the transports of $\s{F} \circ \s{G}$ are:
\[ (\s{F} \circ \s{G})[\sigma](s) = \s{F}[\tau](s). \]
\end{definition}
This means that an $(\s{F} \circ \s{G})$-structure on label set $U$ is formed by arranging the labels into $\s{G}$-structures and arranging the $\s{G}$-structures into an $\s{F}$-structure. The construction of $(\s{F} \circ \s{G})$-structure is illustrated in Figure \ref{fig:compEx}.

\begin{figure}[h!]
\[\begin{tikzpicture}

\wvertex (0) at (0, 0) {W};
\wvertex (1) at (0, -2.5) {O};
\wvertex (2) at (3, -0.5) {L};
\wvertex (3) at (3.5, -1.5) {E};

\wvertex (6) at (1.5, -2) {Y};
\wvertex (7) at (1.5, -4.5) {P};
\wvertex (8) at (3, -3) {T};
\wvertex (9) at (0, -3.5) {M};
\wvertex(10) at (2, -3.5) {X};

\wvertex (4) at (-1, -1.5) {V};
\wvertex (5) at (1.5, -1) {S};

\put (120, 0) {\textcolor{black}{\huge{$\mathcal{F}$}}};

\put (50, 5) {\textcolor{black}{\large{$\mathcal{G}$}}};
\put (-50,-10) {\textcolor{black}{\large{$\mathcal{G}$}}};
\put (-35,-80) {\textcolor{black}{\large{$\mathcal{G}$}}};
\put (90,-120) {\textcolor{black}{\large{$\mathcal{G}$}}};

\draw [ ultra thick] (1, -2) circle(3.5);

\draw [thick] (0, 0.5) to [out=0, in=0] (-1, -2) to [out=180, in=180] (0, 0.5);
\draw [thick] (3, -2.5) to [out=0, in=0] (1.5, -5) to [out=180, in=180] (3, -2.5);
\draw [thick] (0, -2) to [out=0, in=0] (0, -4) to [out=180, in=180] (0, -2);
\draw [thick] (1.5, -2.5) to [out=180, in=225] (1, -1)
to [out=45, in=180] (3, 0) to [out=0, in=45] (4, -1.5)
to [out=225, in=0] (1.5, -2.5);

\end{tikzpicture}\]
\caption{This is a diagram of an arbitrary structure of the species $\s{F}\circ\s{G}$. Given the label set $U = \{W,O,L,V,E,Y,P,T,M,X,V,S\}$, we partition $U$, in this case into the sets $\{W,V\}$, $\{M,O\}$, $\{E,L,S,Y\}$, and $\{P,T,X\}$, put a $\s{G}$ structure on each set and an $\s{F}$ structure on the set of $\s{G}$-stuctures.}
\label{fig:compEx}
\end{figure}
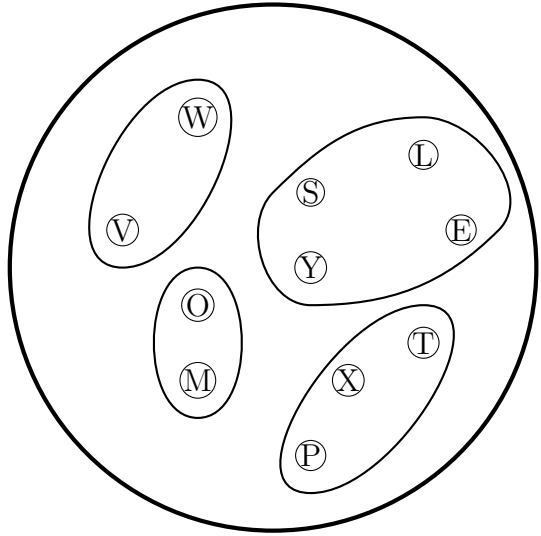

Two $(\s{F} \circ \s{G})$-structures are isomorphic if the two overarching $\s{F}$-structures are isomorphic and the corresponding $\s{G}$-structure parts are isomorphic. The species $\s{G}$ is forbidden from having an empty structure; otherwise, any subdivision of the label set into $\s{G}$-structures would be extendable by adjoining any number of empty $\s{G}$-structures, so $(\s{F} \circ \s{G})$-structures would have arbitrarily many empty $\s{G}$-structures.

Composition of generating series is not as straightforward as their addition and multiplication.
\begin{proposition} \label{prop:sercomp}
For any species $\s{F}$ and $\s{G}$ such that $\s{G}[\emptyset] = \emptyset$,
\[ Z_{\s{F} \circ \s{G}}(p_1, p_2, p_3, \ldots) = Z_{\s{F}}(Z_{\s{G}}(p_1, p_2, p_3, \ldots), \, Z_{\s{G}}(p_2, p_4, p_6, \ldots), \, Z_{\s{G}}(p_3, p_6, p_9), \ldots). \]
A proof of this theorem is given in \cite[Theorem 1.4.2]{1998}. This yields the following two generating series:
\begin{itemize}
\item $(\s{F} \circ \s{G})(x) = \s{F}(\s{G}(x))$;
\item $\iso{(\s{F} \circ \s{G})}(x) = Z_{\s{F}}\left(\iso{\s{G}}(x), \, \iso{\s{G}}(x^2), \, \iso{\s{G}}(x^3), \ldots\right)$.
\end{itemize}
\end{proposition}
The exponential series is literally a composition of the functions. However, it is significant that the expression for the type generating series requires usage of the cycle index series. This is largely why the cycle index series is so useful in enumerating unlabeled structures, as we will see in sections 4 and 5.

\begin{example}
Recall that $\s{E}$ is the set species. For any species $\s{F}$, an $(\s{E} \circ \s{F})$-structure is a set of $\s{F}$-structures. This is particularly meaningful when $\s{F}$ is a species of connected structures. For instance:
\begin{itemize}
\item If $\s{G}$ is the species of graphs and $\s{G}^C$ is the species of connected graphs, then $\s{G} = \s{E} \circ \s{G}^C$.
\item If $\s{S}$ is the species of permutations and $\s{C}$ is the species of (non-empty) cycles, then $\s{S} = \s{E} \circ \s{C}$.
\item If $\s{A}$ is the species of rooted trees, then $\s{A} = X \cdot (\s{E} \circ \s{A})$. A rooted tree can be thought of as a singleton, the root, and a set of rooted trees emanating from the singleton. The root of each of these smaller trees is the vertex in the smaller tree adjacent to the singleton (See Figure \ref{fig:trees}).

\begin{figure}[h!]
\[\begin{tikzpicture}

\node [draw = none] at (0, 4.25) {$X$};
\node [draw = none] at (1.75, 4.25) {\huge{$\cdot$}};
\node [draw = none] at (3.5, 4.25) {$\mathcal{E} \circ \mathcal{A}$};

\wvertex (1) at (0, 0) {D};
\wvertex (2) at (2, 2.25) {R};
\wvertex (3) at (2, 0.75) {A};
\wvertex (4) at (2, -0.75) {G};
\wvertex (5) at (2, -2.25) {O};
\path
(1) edge (2)
(1) edge (3)
(1) edge (4)
(1) edge (5);

\wvertex (6) at (3, 2.25) {N};
\wvertex (7) at (4.5, 3) {E};
\wvertex (8) at (4.5, 2.25) {S};
\wvertex (9) at (4.5, 1.5) {B};
\path
(2) edge (6)
(6) edge (7)
(6) edge (8)
(6) edge (9);

\wvertex (10) at (3, 1.25) {H};
\wvertex (11) at (3, 0.25) {P};
\path
(3) edge (10)
(3) edge (11);

\wvertex (12) at (3, -0.75) {F};
\wvertex (13) at (4, -0.25) {T};
\wvertex (14) at (4, -1.25) {C};
\wvertex (15) at (5, -1.25) {Q};
\path
(4) edge (12)
(12) edge (13)
(12) edge (14)
(14) edge (15);

\draw [thick] (0, 0) circle(0.3);

\draw  (1.5, 2.25) to [out=270, in=180] (3, 1.85)
to [out=0, in=180] (4.5, 1)
to [out=0, in=0] (4.5, 3.5)
to [out=180, in=0] (3, 2.65)
to [out=180, in=90] (1.5, 2.25);

\draw  (1.5, 0.75) to [out=270, in=180] (3, -0.15)
to [out=0, in=0] (3, 1.65)
to [out=180, in=90] (1.5, 0.75);

\draw  (1.5, -0.75) to [out=90, in=180] (3, -0.35)
to [out=0, in=180] (4, 0.25)
to [out=0, in=90] (5.5, -1.25)
to [out=270, in=0] (4, -1.75)
to [out=180, in=0] (3, -1.15)
to [out=180, in=270] (1.5, -0.75);

\draw  (2, -2.25) circle(0.5);

\draw [thick, rounded corners] (1.25, -3) rectangle (5.75, 3.75);

\end{tikzpicture}\]

\caption{An example of the relation $\s{A} = X\cdot\s{E}(\s{A})$. Here the vertex labeled $D$ is singleton and the root, and it is adjacent to $R$, $A$, $G$, and $O$, the roots of the smaller rooted trees.}
\label{fig:trees}
\end{figure}
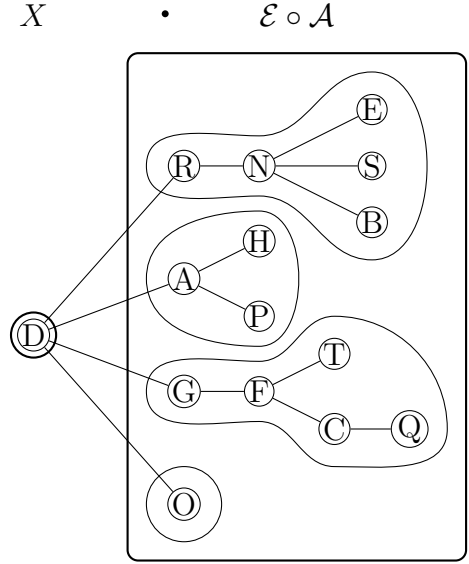

\item Define the species $\textsc{End}$ of endofunctions; that is, $\textsc{End}[U]$ is the set of functions from $U$ to $U$, and
\[ \textsc{End}[\sigma](f) = \sigma f \sigma^{-1}, \]
which is $f$ conjugated by $\sigma$. If $\s{A}$ is the species of rooted trees and $\s{S}$ is the species of permutations, then
\[ \textsc{End} = \s{S} \circ \s{A}. \] To see this, let $G$ be the directed graph (a graph where every edge points towards one of its vertices and we can have up to two edges for each pair of vertices, up to one pointing each way) with vertex set $U$ and an edge from $u$ to $v$ if $f(u)=v$. The reader may note that the function taking $f$ to $G$ is injective. Some elements of $U$ are periodic, and so these create cycles in $G$. The rest of the elements eventually land on a periodic element, and this corresponds to a path in $G$ leading into the vertex of the periodic element. Thus, as seen in Figure \ref{fig:ScircA}, $G$ corresponds to a set of cycles, which we have seen is a permutation, where each vertex is the root of the tree flowing into it. Thus, we have the equation \[\textsc{End} = \s{S}\circ\s{A}.\]

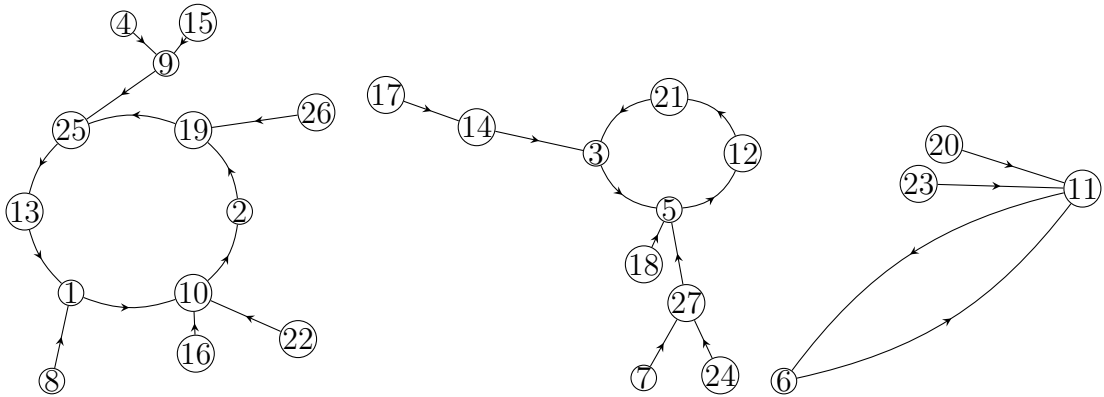
\begin{figure}[h!]
\begin{minipage}[b]{0.3\linewidth}
\[\begin{tikzpicture}[x=1.3cm, y=1cm,
    every edge/.style={
        draw,
        postaction={decorate,
                    decoration={markings,mark=at position 0.5 with {\arrow{>}}}
                   }
        }
]
	\node[draw=none] at (0:0) {};
	\wvertex (1) at (60:1.25) {19};
	\wvertex (2) at (120:1.25) {25};
	\wvertex (3) at (180:1.1) {13};
	\wvertex (4) at (240:1.25) {1};
	\wvertex (5) at (300:1.25) {10};
	\wvertex (6) at (360:1.1) {2};

	\wvertex (9) at (80:2) {9};
	\wvertex (10) at (92:2.5) {4};
	\wvertex (11) at (75:2.6) {15};
	\wvertex (12) at (35:2.3) {26};
	\wvertex (13) at (250:2.4) {8};
	\wvertex (14) at (315:2.4) {22};
	\wvertex (15) at (289:2) {16};

	\path
		(1) edge [>=stealth,bend right=20] (2)
		(2) edge [>=stealth,bend right=15] (3)
		(3) edge [>=stealth,bend right=15] (4)
		(4) edge [>=stealth,bend right=20] (5)
		(5) edge [>=stealth,bend right=20] (6)
		(6) edge [>=stealth,bend right=20] (1)
		
		(9) edge [>=stealth] (2)
		(10) edge [>=stealth](9)
		(11) edge [>=stealth](9)
		(13) edge [>=stealth] (4)
		(12) edge [>=stealth] (1)
		(14) edge [>=stealth] (5)
		(15) edge [>=stealth] (5)
		
	;
\end{tikzpicture}\]
\end{minipage}
\begin{minipage}[b]{0.3\linewidth}
\[\begin{tikzpicture}[x=1.3cm, y=1cm,
    every edge/.style={
        draw,
        postaction={decorate,
                    decoration={markings,mark=at position 0.5 with {\arrow{>}}}
                   }
        }
]
	\node[draw=none] at (0:0) {};
	\wvertex (a) at (0:.75) {12};
	\wvertex (b) at (90:.75) {21};
	\wvertex (c) at (180:.75) {3};
	\wvertex (d) at (270:.75) {5};

	\wvertex (e) at (260:1.5) {18};
	\wvertex (f) at (275:2) {27};
	\wvertex (g) at (280:3) {24};
	\wvertex (h) at (265:3) {7};

	\wvertex (i) at (170:2) {14};
	\wvertex (j) at (165:3) {17};

	\path
		(a) edge [>=stealth, bend right=30] (b)
		(b) edge [>=stealth, bend right=30] (c)
		(c) edge [>=stealth, bend right=30] (d)
		(d) edge [>=stealth, bend right=30] (a)
		
		(h) edge [>=stealth] (f)
		(g) edge [>=stealth] (f)
		(f) edge [>=stealth] (d)
		(e) edge [>=stealth] (d)

		(i) edge [>=stealth] (c)
		(j) edge [>=stealth] (i)
		;
\end{tikzpicture}\]
\end{minipage}
\begin{minipage}[b]{0.3\linewidth}
\[\begin{tikzpicture}[x=1.3cm, y=1cm,
    every edge/.style={
        draw,
        postaction={decorate,
                    decoration={markings,mark=at position 0.5 with {\arrow{>}}}
                   }
        }
]
\node[draw=none] at (0:0) {};
\wvertex (1) at (215:2) {6};
\wvertex (2) at (45:2) {11};

\wvertex (3) at (90:2) {20};
\wvertex (4) at (100:1.5) {23};

\path
(1) edge[>=stealth, bend right=20] (2)
(2) edge [>=stealth, bend right=20] (1)
(3) edge [>=stealth] (2)
(4) edge [>=stealth] (2)
;
\end{tikzpicture}\]
\end{minipage}
\caption{The graph $G$ corresponding to the function $f$ where $f(2)=f(26)=19, f(19)=f(9)=25, f(4)=f(15)=9 f(25)=13, f(13)=f(8)=1, f(1)=f(16)=f(22)=10, f(10)=2, f(3)=f(18)=f(27)=5, f(7)=f(24)=27, f(5)=12, f(12)=21, f(21)=f(14)=3, f(17)=14, f(11)=f(20)=f(23)=6,$ and $f(6)=11$. Note that this construction can be understood as a permutation of rooted trees.}
\label{fig:ScircA} 
\end{figure}

\end{itemize}
\end{example}

We also have the following algebraic facts about composition, the proofs of which follow when considered at the cycle index level:
\begin{proposition}For any species $\s{F}$, $\s{G}$, $\s{H}$ where $\s{H}[\emptyset] = \emptyset$,
\begin{itemize}
\item $(\s{F} + \s{G}) \circ \s{H} = (\s{F} \circ \s{H}) + (\s{G} \circ \s{H})$ (right-distributivity over addition);
\item $(\s{F} \cdot \s{G}) \circ \s{H} = (\s{F} \circ \s{H}) \cdot (\s{G} \circ \s{H})$ (right-distributivity over multiplication);
\item $\s{F} \circ X = \s{F}$ and $X \circ \s{H} = \s{H}$ (identity).
\end{itemize}
\end{proposition}

\subsubsection{Differentiation and pointing}
Next we define the operation of differentiation. Differentiation is most important in that it allows us to root, or \emph{point}, a species. For instance, the species $\s{A}$ of rooted trees can be expressed in terms of the species $a$ of unrooted trees by the formula $\s{A} = X\cdot a'$, as we will see below.

\begin{definition}
Given a species $\s{F}$, define its \emph{derivative} $\s{F}'$ to be the species such that, if $* \not\in U$, then
\[ \s{F}'[U] = \s{F}[U \cup \{ * \}] \]
and
\[ \s{F}'[\sigma](s) = \s{F}[\sigma^+](s), \]
where $\sigma^+(*) = *$ and $\sigma^+(x) = \sigma(x)$ if $x \in U$.
\end{definition}
Thus, we take a derivative by adding a star to the label set and requiring that the star remain fixed by isomorphisms.
This definition ensures that we have the equality $|\s{F}'[n]| = |\s{F}[n+1]|$ for all $n$.

The following proposition explains the name \emph{derivative}. The proof of is left to the reader.
\begin{proposition}
For any species $\s{F}$,
\[ Z_{\s{F}'}(p_1, p_2, p_3, \ldots) = \frac{\partial}{\partial p_1} Z_{\s{F}}(p_1, p_2, p_3, \ldots). \]
This yields the following two generating series:
\begin{itemize}
\item $\s{F}'(x) = \frac{d}{dx} \s{F}(x)$;
\item $\iso{\s{F}'}(x) = \left(\frac{\partial}{\partial p_1}Z_{\s{F}}\right)(x, x^2, x^3, \ldots)$.
\end{itemize}
\end{proposition}

\begin{example}
This following example is accredited to \cite{1998}:
Recall that $\s{L}$ denotes the species of linear orderings. We denote the species of cyclic orderings by $\s{C}$. The generating series for $\s{C}$ can be calculated using species differentiation. It turns out that the generating series of the derivative of cyclic orderings is $\s{C}'(x) = \s{L}(x) = \frac{1}{1-x}$. That is, a $\s{C}'$-structure on a set $U$ is a $\s{C}$-structure on the set $U\cup\{*\}$. Thus, we naturally imagine the derivative of a cylic ordering as a linear ordering, forgetting *, on the set $U$ (depicted in Figure \ref{ceql}).

\begin{figure}[h!]
		\begin{minipage}[b]{0.475\linewidth}
			\[\begin{tikzpicture}
            [x=1cm, y=1cm,
    every edge/.style={
        draw,
        postaction={decorate,
                    decoration={markings,mark=at position 0.5 with {\arrow{>}}}
                   }
        }
]
\node[draw=none] at (0:0) {};
\wvertex (1) at (45:1.5) {1};
\wvertex (2) at (0:1.5) {2};
\wvertex (7) at (-45:1.5) {7};
\wvertex (4) at (-90:1.5) {4};
\wvertex (6) at (-135:1.5) {6};
\wvertex (3) at (-180:1.5) {3};
\wvertex (5) at (135:1.5) {5};
\path
 (1) edge [>=stealth,bend left=20] (2)
 (2) edge[>=stealth,bend left=20] (7)
 (7) edge[>=stealth,bend left=20] (4)
 (4) edge[>=stealth,bend left=20] (6)
 (6) edge[>=stealth,bend left=20] (3)
 (3) edge[>=stealth,bend left=20] (5);
			\end{tikzpicture}\]

		\end{minipage}
		\begin{minipage}[b]{0.05\linewidth}
			\[\begin{tikzpicture}
\put(-10,40){\huge{$\displaystyle = $}}
			\end{tikzpicture}\]

		\end{minipage}
		\begin{minipage}[b]{0.475\linewidth}
			\[\begin{tikzpicture}
                        [x=1cm, y=1cm,
    every edge/.style={
        draw,
        postaction={decorate,
                    decoration={markings,mark=at position 0.5 with {\arrow{>}}}
                   }
        }
]
\node[draw=none] at (0:0) {};
\wvertex (1) at (45:1.5) {1};
\wvertex (2) at (0:1.5) {2};
\wvertex (7) at (-45:1.5) {7};
\wvertex (4) at (-90:1.5) {4};
\wvertex (6) at (-135:1.5) {6};
\wvertex (3) at (-180:1.5) {3};
\wvertex (5) at (135:1.5) {5};
\node[draw=none] (*) at (90:1.5) {\huge{*}};

\path
(1) edge[>=stealth, bend left=20] (2)
(2) edge[>=stealth, bend left=20] (7)
(7) edge[>=stealth, bend left=20] (4)
(4) edge[>=stealth, bend left=20] (6)
(6) edge[>=stealth, bend left=20] (3)
(3) edge[>=stealth, bend left=20] (5)
(5) edge[>=stealth, bend left=20] (*)
(*) edge[>=stealth, bend left=20] (1);

			\end{tikzpicture}\]
		\end{minipage}
\caption{$\s{L} = \s{C}'$}\label{ceql}
	\end{figure}
    
Therefore, in terms of its generating series, $$\s{C}(x) = \int_0^x \frac{dx}{1-x} = \log\frac{1}{1-x}.$$
\end{example}

Now we define the operation of pointing. The effect of pointing is to take an unrooted structure and turn it into a rooted structure.

\begin{definition} Given a species $\mathcal{F}$, the species of \emph{pointed $\mathcal{F}$-structures} is the species $X\cdot \mathcal{F}'$.\end{definition}

\begin{example} If $\mathcal{A}$ is the species of rooted trees, and $a$ is the species of unrooted trees, then $a'$ is the species of unrooted trees where one vertex is replaced by a $*$. Thus, in an $X\cdot a'$-structure, we can connect the singleton vertex to the $*$, and interpret this as removing the $*$ and connecting the singleton to the vertices that were connected to the $*$. The result is a rooted tree, so we have the equation $\s{A} = X\cdot a'$.\end{example}

\subsection{Virtual species}

Once we have defined the addition and multiplication operators on species, we may ask whether we can define their inverses. We thus define the concept of a virtual species as a subtraction of two species. This definition is essentially the same as the definition of $\Z$ in terms of $\N$, and allows us to define division.

Addition, multiplication, and composition work on virtual species as they do on regular species. Therefore, instead of worrying about the technical definitions, it would benefit the reader to understand that these definitions represent an extension of the usual species operations onto virtual species.

\begin{definition} Let $\s{F}$ and $\s{G}$ be species. Then we define the \emph{virtual species} $\s{F}-\s{G}$ as the element $(\s{F},\s{G})$ in $\{(\s{A},\s{B}) : \s{A}\text{ and } \s{B} \text{ are species }\}/\sim$, where $\sim$ is the equivalence relation \[(\s{A},\s{B})\sim (\s{C},\s{D}) \iff \s{A}+\s{D} \equiv \s{B}+\s{C}.\] The \emph{additive inverse} of $\s{F}$, denoted $-\s{F}$, is given as $-\s{F} = \mathbf{0} - \s{F}$.
\end{definition}

Note that in the case that $\s{B}$ is a subspecies of $\s{A}$, $\s{A}-\s{B}$ may be thought of as the ordinary species of $\s{A}$ structures that are not $\s{B}$ structures.

The set of virtual species is a commutative ring under addition and multiplication.

Using our new rule for subtraction, we can define a multiplicative inverse for most species. Given a species $\s{F}$ with one structure on the empty set, we define \[\frac{\mathbf{1}}{\s{F}} = \frac{\mathbf{1}}{\mathbf{1}+\s{F}_+} = \sum_{n=0}^\infty (-1)^n (\s{F}_+^n).\] Note that we can only make this definition if the family $\s{F}_+^n$ is summable---that is, if for any set $U$, $\s{F}_+^n = \emptyset$ for all $n$ greater than some $N$.

We can also define composition on virtual species. The technical definition \cite[p.~127]{1998} is that, given virtual species $\Phi$ and $\Psi$ with $\Psi = \s{H}-\s{K}$ and $\Psi[\emptyset] = \emptyset$, \[\Phi\circ\Psi = \Phi(X+Y) \times (E(X)E^{-1}(Y))|_{X:=H, Y:=K}.\]

\subsubsection{The species logarithm \texorpdfstring{$\Omega$}{Omega}}

Let $\Omega$ be the compositional inverse of $\s{E}_+$; that is, $\Omega$ is the virtual species that satisfies the equation \[\s{E}_+\circ \Omega = \Omega\circ \s{E}_+ = X.\] It is shown in \cite{1998} that such a species $\Omega$ exists and is unique, and \cite{LaBelle} provides insight into how its cycle index series can be calculated.

Given any species $\s{A}$ with $\s{A}[\emptyset] = \emptyset$, the species $\s{B} = \Omega\circ \s{A}$ is known as the \emph{combinatorial logarithm} of $\s{A}$. Note that this means that $\s{A} = \s{E}_+\circ \s{B}$.

Often $\Omega\circ \s{A}$ can be thought of as connected $\s{A}$-structures. For instance, the species $\Omega\circ \s{G}$ is the species of connected graphs, $\s{G}^C$. This is true because arbitrary graphs are disjoint collections of connected graphs. This gives us the equation $\s{G} = \s{E} \circ \s{G}^C$, so by left-composing $\Omega$, we have $\s{G}^C = \Omega\circ \s{G}$.

In some cases, there is no concept of connected $A$-structures, and for some species, like linear orders, the species $\Omega\circ \s{L}$ is strictly virtual (has negative terms). However, it is important to notice that there are many species $\s{F}$ for which $\Omega \circ \s{F}$ is positive (has no negative terms), such as the species of graphs $\s{G}$ as shown above.

The reason for the name ``combinatorial logarithm'' comes from the effect that composition with $\Omega$ has on the associated series of a species. For instance if $\s{F}(x)$ is the exponential generating series for $\s{F}$, then \cite[131]{1998} \[(\Omega\circ \s{F})(x) = \log \s{F}(x).\] Similar but more complicated formulas hold for the type generating series and the cycle index series.

\subsection{\texorpdfstring{$\Gamma$}{Gamma}-species and quotient species}
Up until now, we have seen several species operations and have introduced the concept of virtual species and the combinatorial logarithm (work accredited to \cite{1998}). We utilize these tools to build other species out of species we already know how to enumerate. However, these methods do not encompass all that we can do to find other species. Some species have structures that are best described as orbits of another species' structures under some group action (quotient species). But, before getting to the orbits themselves, we describe how a group can act on the structures ($\Gamma$-species). The following concepts of $\Gamma$-species and quotient species can be found in \cite{dissertation} and \cite{Andrew_Paper}.

\subsubsection{\texorpdfstring{$\Gamma$}{Gamma}-species}

So far we have been applying actions on the labels of structures. But, what if we wanted to apply actions, specifically group actions, on the structures themselves? This brings us to the definition of $\Gamma$-species.

\begin{definition}
Let $\Gamma$ be a group. Then a {\emph{$\Gamma$-species}} $\s{F}$ is a combinatorial species $\s{F}$ equipped with a group action of $\Gamma$ on $\s{F}$-structures that respects isomorphisms between $\s{F}$-structures.
\end{definition}

\begin{example} \label{ex:graphexample} Consider $\s{G}$, the species of graphs. Let the group $S_2$ act on $\s{G}$-structures (that is, graphs) by sending each graph $G$ to itself via the identity and sending $G$ to its complement $G^C$ via the group element $(12)$. It is easily checked that this group action respects graph isomorphisms; thus, $\s{G}$ with this action of $S_2$ is an $S_2$-species.
\end{example}

To complete our introduction to $\Gamma$-species, we define the $\Gamma$--cycle index series.

\begin{definition}
For a group $\Gamma$ and a $\Gamma$-species $\s{F}$, the \emph{cycle index series of $\s{F}$} is defined as
$$Z_{\s{F}}^{\Gamma}(\gamma)(p_1,p_2,p_3,\cdots) = \sum_{n=0}^{\infty}\frac{1}{n!}\sum_{\sigma\in S_n}\fix(\gamma\cdot \s{F}[\sigma])p_{\sigma}$$

where $p_\sigma = p_1^{\sigma_1}p_2^{\sigma_2}p_3^{\sigma_3}\cdots p_n^{\sigma_n}$ and $\gamma \in \Gamma$.

\end{definition}

This analog to the cycle index series for species takes an additional paramter, $\gamma$, which is necessary to include since equipping the group action to an $\s{F}$-structure may alter whether or not it stays fixed by our permutation. Similar to the cycle index series for species, this cycle index series counts the $\s{F}$ structures according to the permutation data, but now also with respect to the group actions.

Note that if $\s{F}$ is a $\Gamma$-species, the ordinary cycle index of $\s{F}$ can be found by evaluating the $\Gamma$ cycle index series at the identity element of $\Gamma$. 

Like the algebraic operations we use on ordinary species, we define similar operations on the cycle index series of $\Gamma$-species. This allows us to build even more species using $\Gamma$-species.

\noindent {\bf Addition:}
For two $\Gamma$-species $\s{F}$ and $\s{G}$, the $\Gamma$--cycle index for their sum is
$$Z_{\s{F}+\s{G}}^{\Gamma}(\gamma) = Z_{\s{F}}^{\Gamma}(\gamma) + Z_{\s{G}}^{\Gamma}(\gamma).$$

\noindent {\bf Multiplication:} For two $\Gamma$-species $\s{F}$ and $\s{G}$, the $\Gamma$--cycle index for their product is
$$Z_{\s{F}\cdot \s{G}}^{\Gamma}(\gamma) = Z_{\s{F}}^{\Gamma}(\gamma) \cdot Z_{\s{G}}^{\Gamma}(\gamma).$$

Now, before defining the $\Gamma$--cycle index for two $\Gamma$-species' composition, we need to define the following terms.
\begin{definition} \label{def:gammaComposition}
The \emph{composition} for two $\Gamma$-species $\s{F}$ and $\s{G}$ is the $\Gamma$-species $\s{F}\circ \s{G}$ defined as
$$(\s{F}\circ \s{G}) [A] = \prod_{\pi\in P(A)}\left(\s{F}[\pi]\times \prod_{B\in\pi}\s{G}[B]\right),$$ where $P(A)$ is the set of partitions of $A$ and where $\gamma\in\Gamma$ acts on an $(\s{F}\circ \s{G})$-structure by acting on the $\s{F}$-structure and the $\s{G}$-structures independently.
\end{definition}
\begin{definition} \label{def:gammaPlethysm}
The \emph{plethysm} of two $\Gamma$--cycle indices $f$ and $g$, $f\circ g$, is $$(f\circ g) (\gamma) = f(\gamma)(g(\gamma)(p_1, p_2, p_3, ...), g(\gamma^2)(p_2, p_4, p_6, ...), ...),$$ a $\Gamma$--cycle index.
\end{definition}

\noindent {\bf Composition:} For two $\Gamma$-species $\s{F}$ and $\s{G}$ where $\s{G}[\emptyset] = \emptyset$, the $\Gamma$--cycle index for their composition is $$Z_{\s{F}\circ \s{G}}^{\Gamma}(\gamma) = Z_{\s{F}}^{\Gamma}(\gamma) \circ Z_{\s{G}}^{\Gamma}(\gamma).$$

\subsubsection{Quotient species}
We now introduce quotient species, which is a way to get an ordinary species from a $\Gamma$-species. Conviently, passing a $\Gamma$-species through its quotient allows us to discover and enumerate several species. For instance, suppose we wanted to count the species of pairs of complementary graphs. On the surface, this process seems difficult to control. However, we may pass the $\Gamma$-species from Example \ref{ex:graphexample} through its quotient to get and then enumerate our desired species. The following definitions allow us to enact such a process.

\begin{definition}
Let $\Gamma$ be a group and $\s{F}$ be a $\Gamma$-species. We say the \emph{$\Gamma$-orbit of an $\s{F}$-structure} $s$ is the orbit of $s$ under the group action of $\Gamma$.
\end{definition}

\begin{definition}
Let $\s{F}$ be a $\Gamma$-species with a group action given. We define $\s{F}/\Gamma$, the \emph{quotient species} of $\s{F}$ under the group action of $\Gamma$, to be the species of $\Gamma$-orbits on $\s{F}$-structures.
\end{definition}

A quotient species is in fact a species, since it is a rule that satisfies the required conditions discussed in definition \ref{def:spec}. A more rigorous approach as to why quotient species is a species can be found in \cite{Bousquet}.

So returning to Example \ref{ex:graphexample}, by passing the $S_2$-species $\s{G}$ through its quotient, we put a given graph $G$ and its complement $G^C$ into an equivalence class. Thus, we have effectively constructed the species of pairs of complements as $\s{G}/S_2$. Using technology, we can enumerate this species via its cycle index series.

\begin{theorem} \label{thm:quotspeciescis}
For a $\Gamma$-species $\s{F}$, the \emph{ordinary cycle index series of the quotient species $\s{F}/\Gamma$} is $$Z_{\s{F}/\Gamma} = \frac{1}{|\Gamma|}\sum_{\gamma\in\Gamma}Z_{\s{F}}^{\Gamma}(\gamma) = \frac{1}{|\Gamma|} \sum_{n=0}^{\infty} \frac{1}{n!} \sum_{\sigma\in S_n}\sum_{\gamma\in\Gamma} \fix (\gamma\cdot \s{F}[\sigma])p_{\sigma}.$$
\end{theorem}
A proof of \ref{thm:quotspeciescis} can be found in \cite{dissertation}. Note that $Z_{\s{F}\Gamma}$ is the average over the group elements $\gamma\in\Gamma$ of $Z_{\s{F}}^\Gamma(\gamma)$. This formula comes from Burnside's Lemma.

\section{Enumerating point-determining bipartite graphs}

The enumeration of point-determining bipartite graphs appears to be absent from the literature. With the use of species theory, and especially the tool of $\Gamma$-species, this problem is relatively quick to solve. We use this enumeration as an example of the power of these tools, as well as an example of how to use Sage \cite{Sage} in species calculations found in Appendix \ref{appB}.

We start with the assumption that we can compute the cycle index series for $\s{BC}$ and $\Omega$ since the former is done in 	\cite{dissertation} and the latter is done in \cite{LaBelle}.

\begin{definition} \label{def:pd} Given a graph $G$, recall that the neighborhood of a vertex $v\in V(G)$ is the set of vertices to which $v$ is connected by an edge. A \emph{point-determining} graph is a graph where no two vertices share a neighborhood.\end{definition}

\begin{example} The complete graph on $n$ vertices, $K_n$, is point-determining since the neighborhood of each vertex is all of the vertices except itself. Any graph with two or more isolated vertices is not point-determining since each isolated vertex has an empty neighborhood.\end{example}

Because the number of ways to bicolor a bipartite graph depends on the number of connected components of the graph, we will pass to connected graphs on the way to our enumeration, which will allow us to perform a quotient operation. At the end of the example, we will pass back to not-necessarily-connected graphs for the final result.

We will refer to the species of point-determining graphs as $\s{P}$, the species of bipartite graphs as $\s{BP}$, the species of bicolored graphs as $\s{BC}$, the species of connected graphs as $\s{C}$, and any combination of these properties as the concatenation of these symbols. For instance, the species of connected point-determining bipartite graphs we will refer to as $\s{CPBP}$.

We start our computations by noticing a relationship between point-determining graphs and arbitrary graphs. We phrase the lemma in terms of bipartite graphs, but the result holds for many classes of graphs, including the class of all graphs (See Section \ref{sec:pdPhi}).

\begin{lemma} \label{lem:firstlemma} $\s{PBP} = \s{BP}\circ\Omega$.\end{lemma}

\begin{proof} Consider a bipartite graph $G$. This graph (potentially) has multiple vertices with the same neighborhood. So, to construct a bipartite graph from $G$, we need to take equivalence classes of vertices with the same neighborhood. To look at it from the other direction, consider a point-determining bipartite graph $P$. Every vertex in $P$ has a unique neighborhood, but an arbitrary bipartite graph $G$ can have many vertices sharing the same neighborhood. Therefore, each vertex in $P$ corresponds to a nonempty set of vertices in $G$, each of which has the same neighborhood as the original vertex. This clearly respects transports, so $\s{BP} = \s{PBP}\circ \s{E}_+$. By right-composing $\Omega$ we get the desired result.\end{proof}

Next, we perform a quotient operation.

\begin{lemma} \label{lem:secondlemma} $\s{CBP} = \s{CBC}/S_2$.\end{lemma}

\begin{proof} Let the action of $S_2$ on $\s{CBC}$ be the color-flipping operation, which flips the color of each vertex to the opposite color. Note that this action on a bicolored graph produces another bicolored graph, and that the action is an involution. For any $\s{CBC}$-graph $G$, let $H$ be the image of $G$ under the $S_2$ action, and note that by removing the colors from $G$ and $H$, we get two copies of the same bipartite graph $P$. By Proposition \ref{thm:bicoloredways}, $P$ has exactly 2 proper bicolorings, so the $S_2$-orbit of $G$ is exactly the set of bicolored graphs that produce $P$ when we remove the colors. Thus, we can associate $P$ with the $S_2$-orbit of $G$. This clearly respects transports, so the result holds.\end{proof}

Now, we can use these lemmas to find an equation for $\s{PBP}$ in terms of the known species $\s{E}$, $\Omega$, and $\s{BP}$.

\begin{lemma} \label{lem:pbp}
$\s{PBP} = (\s{E}\circ ((\Omega\circ \s{BC}_+)/S_2))\circ \Omega$.\end{lemma}

\begin{proof} By Lemma \ref{lem:firstlemma}, $\s{PBP} = \s{BP}\circ\Omega$. Since $\s{CBP}$ is connected bipartite graphs, $\s{BP} = \s{E}\circ \s{CBP}$. By Lemma \ref{lem:secondlemma}, $\s{CBP} = \s{CBC}/S_2$. Finally, a nonempty bicolored graph is a nonempty set of connected bicolored graphs, so by left-composing $\Omega$, we get $\s{CBC} = \Omega\circ \s{BC}_+$. Combining all of this, we get the desired result.\end{proof}

We can use this formula and the cycle index series for $\s{BC}$ and $\Omega$ to compute the associated series of $\s{PBP}$ (Appendix \ref{appB}).

\section{Point-determining graphs and graphs without endpoints}

In this section, we generalize some results of Gessel and Li \cite{2007} on point-determining graphs and graphs without endpoints. This generalization is significant in large part because it yields enumerative results for certain subspecies of bipartite graphs and connected bipartite graphs.

We will use the following notation throughout this section.

\begin{definition}
If $\Phi$ is a subspecies of the graph species $\s{G}$, then we will say that $\Phi$ is a \emph{graph subspecies}; a \emph{$\Phi$-graph} is a graph belonging to the species $\Phi$ (that is, a $\Phi$-structure).
\end{definition}

\subsection{Point-determining \texorpdfstring{$\Phi$}{Phi}-graphs} \label{sec:pdPhi}

First, some definitions and notation. Recall from Section 4 the definition of point-determining:
\begin{definition}
Given a graph $G$ and a vertex $v$, the \emph{neighborhood} of $v$ is the set of vertices of $G$ that are adjacent to $v$. A graph is \emph{point-determining} if no two distinct vertices have the same neighborhoods. The species of point-determining graphs is called $\s{P}$; if $\Phi$ is a graph subspecies, then the species of point-determining $\Phi$-graphs is called $\Phi_{\s{P}}$.
\end{definition}
We can also think of point-determining graphs in terms of duplicate vertices:
\begin{definition}
Let $G$ be a graph and let $v$ and $w$ be distinct vertices of $G$. Then $v$ and $w$ are \emph{duplicates} if they have the same neighborhoods.
\end{definition}
Then a graph is point-determining if and only if none of its vertices has a duplicate.

Gessel and Li \cite[Theorem~2.2]{2007} prove the following theorem, which we present using our own notation and wording. We will generalize this theorem to apply to other classes of graphs.

\begin{theorem} \label{thm:original1}
If $\s{G}$ is the species of all graphs, $\s{P}$ is the species of point-determining graphs, and $\s{E}_+$ is the species of nonempty sets (or nonempty edgeless graphs), then
\[ \s{G} = \s{P} \circ \s{E}_+. \]
\end{theorem}

The proof of our generalized version, as well as the version by Gessel and Li, requires a combinatorial interpretation of the composition of two graph subspecies:
\begin{definition}
Let $G$ be a graph with vertices $v_1, v_2, \ldots, v_n$; let $H_1, H_2, \ldots, H_n$ be pairwise disjoint graphs, and for each $k$ let $V_k$ be the vertex set of $H_k$. Then the \emph{superimposition} of $G$ on the graphs $H_1, \ldots, H_n$ is the graph with vertex set $V_1 \cup V_2 \cup \cdots \cup V_n$ where vertices $u \in V(H_i)$ and $v \in V(H_j)$ are adjacent if either
\begin{itemize}
\item $v_i$ and $v_j$ are adjacent in $G$ or
\item $i = j$ and $u$ and $v$ are adjacent in $H_i$.
\end{itemize}
If $\Phi$ and $\Psi$ are two graph subspecies, then $\Phi \diamond \Psi$ denotes the species consisting of superimpositions of a $\Phi$-graph on a set of $\Psi$-graphs.
\end{definition}
This means that we are taking a collection of graphs and linking their vertices with edges according to the superimposed graph. The following lemma (first stated by \cite[Lemma~1.4]{2007}) relates this idea with species composition:
\begin{lemma} \label{lem:super}
If $\Phi$ and $\Psi$ are two graph subspecies such that every $(\Phi \diamond \Psi)$-graph can be expressed uniquely as a superimposition of a $\Phi$-graph on a set of $\Psi$-graphs, then $\Phi \diamond \Psi = \Phi \circ \Psi$.
\end{lemma}

In their proof of Theorem \ref{thm:original1}, Gessel and Li ~\cite{2007} show the following fact:
\begin{proposition} \label{prop:remark} A graph $G$ is the superimposition of a point-determining graph $H$ on a set of edgeless graphs if and only if $H$ is the graph obtained from $G$ by removing all duplicates of every vertex.
\end{proposition}
We rely on this in the proof of our generalized result. There is one condition on the graph subspecies $\Phi$ that allows the generalized theorem on point-determining graphs to hold:
\begin{definition}
Let $\Phi$ be a graph subspecies. We say that $\Phi$ is \emph{closed under creating and deleting duplicates of vertices} if, for every graph $G$ and any distinct vertices $v$ and $w$ that are duplicates of each other, the following condition holds:
\[ \mbox{$G$ is a $\Phi$-graph if and only if $G - v$ is a $\Phi$-graph.}\]
\end{definition}
We are now ready for the theorem.

\begin{theorem} \label{thm:pd}
If $\Phi$ is a graph subspecies that is closed under creating and deleting duplicates of vertices, then $\Phi = \Phi_{\s{P}} \circ \s{E}_+$.
\end{theorem}

\begin{proof}
We first prove that $\Phi_{\s{P}} \diamond \s{E}_+ = \Phi_{\s{P}} \circ \s{E}_+$, using Lemma \ref{lem:super}. Let $G$ be a $(\Phi_{\s{P}} \diamond \s{E}_+)$-graph; then it is the superimposition of a $(\Phi_{\s{P}})$-graph $H$ on a set of edgeless graphs. By Proposition \ref{prop:remark}, $G$ equals a unique such superimposition. Therefore, by Lemma \ref{lem:super}, $\Phi_{\s{P}} \diamond \s{E}_+ = \Phi_{\s{P}} \circ \s{E}_+$.

We now prove that $\Phi = \Phi_{\s{P}} \diamond \s{E}_+$. Since both of these species are subspecies of the species of graphs $\s{G}$, it is sufficient to prove that their sets of structures are equal.

Let $G$ be a $\Phi$-graph. By Proposition \ref{prop:remark}, $G$ is a superimposition of a point-determining graph $H$ on a set of edgeless graphs, where $H$ is obtained from $G$ by deleting duplicates. Since $\Phi$ is closed under deleting duplicates, we have that $H$ is a $\Phi$-graph, and so it is a $\Phi_{\s{P}}$-graph. Therefore, $G$ is the superimposition of a $\Phi_{\s{P}}$-graph on a set of edgeless graphs; that is, $G$ is a $(\Phi_{\s{P}} \diamond \s{E}_+)$-graph.

Now let $G$ be a $(\Phi_{\s{P}} \diamond \s{E}_+)$-graph. Then $G$ is a superimposition of a $(\Phi_{\s{P}})$-graph $H$ on a set of edgeless graphs. By Proposition \ref{prop:remark}, $G$ can be obtained from $H$ by creating duplicates. So, since $H$ is a $\Phi$-graph and $\Phi$ is closed under creating duplicates, $G$ is a $\Phi$-graph.

Therefore, $\Phi = \Phi_{\s{P}} \diamond \s{E}_+ = \Phi_{\s{P}} \circ \s{E}_+$.
\end{proof}

Gessel and Li's result \cite[Theorem~2.2]{2007} (shown here as Theorem \ref{thm:original1}) is obtained as the special case $\Phi = \s{G}$. Here are some other special cases.
\begin{corollary} \label{cor:pd} Let $r \ge 2$ be an integer; $\Phi = \Phi_{\s{P}} \circ \s{E}_+$ if $\Phi$ is one of these graph subspecies:
\begin{itemize}
\item $_r\s{G}$, the species of $r$-partite graphs;
\item $\s{G}^C_{\ge2}$, the species of connected graphs on at least two vertices;
\item $_r\s{G}_{\ge2}^C$, the species of connected $r$-partite graphs on at least two vertices.
\end{itemize}
\end{corollary}

In practice, we can use Theorem \ref{thm:pd} to solve for $\Phi_{\s{P}}$ in terms of $\Phi$:
\begin{corollary} \label{cor:psolve}
If $\Phi$ is a graph subspecies that is closed under creating and deleting duplicates of vertices, then $\Phi_{\s{P}} = \Phi \circ \Omega$.
\end{corollary}
Recall that $\Omega$ is a virtual species and is the compositional inverse of $\s{E}_+$. In the case where $\Phi = \s{BP}$ (bipartite graphs), this corollary yields an expression for $\s{PBP}$ (point-determining bipartite graphs), which is the same result we have in Lemma \ref{lem:firstlemma}.

\subsection{\texorpdfstring{$\Phi$}{Phi}-graphs without endpoints}
In Section 3 of ~\cite{2007}, Gessel and Li discuss graphs without endpoints. Their work focused on connected graphs; we will now generalize Gessel and Li's work to $\Phi$-graphs.

\begin{definition}
Given a general graph $G$, a vertex $v$ of $G$ is an \emph{endpoint} of $G$ if and only if $d(v)=1$ (where $d(v)$ is defined as the degree of vertex $v$).
\end{definition}

Notice that the complete graph on two vertices has two endpoints, whereas the complete graph on three vertices has zero endpoints. This leads us to our next two definitions. Note that the empty graph is not considered a connected graph.

\begin{definition}
A graph $G$ is a \emph{graph without endpoints} if, for every vertex $v$ of $G$, $d(v)\neq 1$. The species of graphs without endpoints is called $\s{M}$; if $\Phi$ is a graph subspecies, then the species of $\Phi$-graphs without endpoints is called $\Phi_{\s{M}}$. 
\end{definition}

\begin{definition}
Given a connected graph $G$, the induced subgraph $C\subseteq G$ is a \emph{core} of $G$ if and only if $C$ is a maximal subgraph without endpoints.
\end{definition}
The core of a tree on $\ge 2$ vertices is not unique: every vertex is a core. As for non-tree graphs, it is not immediately apparent whether its core is unique, or even whether it has a core. The following lemma addresses that.
\begin{lemma} \label{lem:core}
If a connected graph $G$ is not a tree, then it has a \textit{unique} core, obtained by removing the endpoints and their edges until there are no endpoints left. This core has at least two vertices.
\end{lemma}
\begin{proof}
Let $G$ be a connected graph that is not a tree. We first prove existence of a core. If $G$ has no endpoints, then it is already a core of itself, and we are done; otherwise, remove an endpoint (and its edge) from $G$. Continue to remove endpoints one by one until none remain, resulting in a graph $G'$ without endpoints. Since removing an endpoint cannot disconnect the graph, $G'$ is connected. Therefore, $G'$ is a connected induced subgraph of $G$ without endpoints, and so there is a maximal such subgraph. Therefore there is a core.

We now prove that every core of $G$ must have at least two vertices. Since $G$ is not a tree, it has a cycle, which must have at least $3$ vertices. The induced subgraph on this cycle has no endpoints. So every core, being a \textit{maximal} induced subgraph without endpoints, must have at least 3 vertices. (And $3 \ge 2$.)

We now prove uniqueness. Suppose for contradiction that $G$ has two different cores $C_1$ and $C_2$. Each of these is connected and must have at least two vertices, so every vertex of $C_1$ (resp.~$C_2$) has degree $\ge 2$ in $C_1$ (resp.~$C_2$). Now let $C$ be the induced subgraph on the vertices of $C_1 \cup C_2$. Every vertex of $C$ had degree $\ge 2$ before the union, and taking the induced subgraph on the union cannot reduce the degree, so it still has degree $\ge 2$ in $C$ after the union. Thus, $C$ is an induced subgraph of $G$ without endpoints, and it has more vertices than either of $C_1$ and $C_2$, contradicting the fact that $C_1$ and $C_2$ were maximal. Therefore, $G$ has a unique core.
\end{proof}

There is one condition on the graph subspecies $\Phi$ that allows the generalized theorem on graphs without endpoints to hold:
\begin{definition}
Let $\Phi$ be a graph subspecies. We say that $\Phi$ is \emph{closed under creating and deleting endpoints} if, for every graph $G$ and any endpoint $v$ of $G$, the following condition holds:
\[ \mbox{$G$ is a $\Phi$-graph if and only if $G - v$ is a $\Phi$-graph.}\]
\end{definition}
We are now ready for the theorem.

\begin{theorem} Let $\Phi$ be a subspecies of graphs in which every graph is connected, and assume that $\Phi$ is closed under creating and deleting endpoints. If $\s{A}$ is the species of rooted trees, $a$ the species of unrooted trees, and $\Phi_{\s{M}_{\geq2}}$ the species of $\Phi_{\s{M}}$-graphs on at least two vertices, then
   $$\Phi = \left\{
     \begin{array}{ll}
       \Phi_{\s{M}_{\geq2}} \circ \s{A} + a & \text{if the one-vertex graph is a $\Phi$-graph;}\\
       \Phi_{\s{M}} \circ \s{A} & \text{otherwise.}
     \end{array}
   \right.$$
\end{theorem}

\begin{proof}
Let $G$ be a $\Phi$-graph.

If the one-vertex graph is a $\Phi$-graph, then $G$ is either a tree or has a cycle. If $G$ has a cycle, then by Lemma \ref{lem:core}, $G$ has a unique core of two or more vertices. Therefore, $S$ cannot be $G$'s core but, any $\Phi_{\s{M}_{\geq 2}}$-graph is a contender to be the core of $G$. In order to obtain the endpoints of $G$ from its core, we root the necessary trees at vertices in the core of $G$. This is a $(\Phi_{\s{M}_{\geq2}} \circ \s{A})$-structure. We can get all such structures this way, because $\Phi$ is closed under creating and deleting endpoints. This respects transports because $\Phi$ is a graph subspecies.

If, on the other hand, $G$ is a tree, then it does not have a unique core. Thus, we construct the $\Phi$-graph $G$ just by counting it as a tree. This is an $a$-structure. We can get all $a$-structures this way, because we can start with the one-vertex graph (which is a $\Phi$-graph) and create endpoints to build any tree as a $\Phi$-graph. Combining the two cases, we get the sum $\Phi = \Phi_{\s{M}_{\geq2}} \circ \s{A} + a$, as claimed.

If the one-vertex graph is not a $\Phi$-graph, then we claim that $G$ is not a tree. If it were, then we could delete the endpoints of $G$ until we are left with the one-vertex graph. Since $\Phi$ is closed under deleting endpoints, the one-vertex graph must be a $\Phi$-graph, a contradiction. Thus, $G$ is not a tree.

So $G$ contains a cycle. Then, by the same reasoning as the first case, we get that $G$ is a $(\Phi_{\s{M}} \circ \s{A})$-structure.
\end{proof}

In practice, we can use this to solve for $\Phi_{\s{M}}$ in terms of $\Phi$:
\begin{corollary} \label{cor:phi}
Let $\Phi$ be a subspecies of graphs in which every graph is connected, and assume that $\Phi$ is closed under creating and deleting endpoints. Then
\[ 
\Phi_{\s{M}} = \left\{
     \begin{array}{lr}
       \Phi \circ \ainv - \s{E}_2 + X^2 & \text{if the one-vertex graph is a $\Phi$-graph;}\\
       \Phi \circ \ainv & \text{otherwise}
     \end{array}
   \right. \]
where $\ainv$ is the compositional inverse of $\s{A}$ (rooted trees) and $\s{E}_2$ is the species of two-element sets. \end{corollary}
\begin{proof}  If $\Phi$ does not contain the singleton vertex, we have:
 \begin{align*}
 \Phi &= \Phi_{\s{M}} \circ \s{A} \\
 \Rightarrow \Phi_{\s{M}} &= \Phi \circ \ainv.
 \end{align*}
 
 If $\Phi$ contains the singleton vertex: 
 \begin{align*}
 \Phi &= \Phi_{\s{M}_{\geq2}} \circ \s{A} + a \\
 \Rightarrow \Phi_{\s{M}} &= (\Phi - a) \circ \ainv + X \\
 &= \Phi\circ \ainv - (a \circ \ainv) + X.
 \end{align*}
 
 Now we use the dissymmetry theorem for trees, a proof of which is in \cite[Theorem 4.1.1]{1998}.
 \begin{align*}
 \s{A} + \s{E}_2 \circ \s{A} &= a + \s{A}^2 \\
 \Rightarrow a &= \s{A} + \s{E}_2\circ \s{A} - \s{A}^2\\
 \Rightarrow a &= (X + \s{E}_2 - X^2) \circ \s{A}.
 \end{align*}
 So we have the formula for $\Phi_\s{M}$:
 \begin{align*}
 \Phi_\s{M} &= \Phi\circ \ainv - (X + \s{E}_2 - X^2) \circ \s{A} \circ \ainv + X \\
 &= \Phi \circ \ainv - (X + \s{E}_2 - X^2) \circ X + X \\
 &= \Phi\circ \ainv - X - \s{E}_2 + X^2 + X \\
 &= \Phi\circ \ainv - \s{E}_2 + X^2.
 \end{align*}
\end{proof}

Notice that if $\Phi$ is the species of connected graphs ($\s{G}^C$), then this corollary yields $\s{G}^C_{\s{M}} = \s{G}^C \circ \ainv - \s{E}_2 + X^2$. This is the exact result found in \cite{2007}: Gessel and Li show that $\s{M}^C = \s{G}^C \circ (X\s{E}(-X)) + \s{E}_2(-X)$. The proofs that $(X\s{E}(-X)) = \ainv$ and that $-\s{E}_2 + X^2 = \s{E}_2(-X)$ are left to the reader.

\subsection{Correspondence between point-determining \texorpdfstring{$\Phi$}{Phi}-graphs and \texorpdfstring{$\Phi$}{Phi}-graphs without endpoints}

Gessel and Li prove that $\iso{\mathcal{P}^C}(x) = \iso{\mathcal{M}^C} + x^2$ \cite[cor.~3.2]{2007}. That is, when counting the unlabeled graphs on $n$ vertices, the number of connected point-determining graphs is equal to the number of connected graphs without endpoints (except for $n=2$). We now prove a more general theorem, whose proof parallels that of Gessel and Li's result.

\begin{theorem}
Let $\Phi$ be a species of connected graphs containing the one-vertex graph, and let $\Phi_{\ge 2} = \Phi - X$ be the species of $\Phi$-graphs on at least two vertices. If $\Phi$ is closed under creating and deleting endpoints and $\Phi_{\ge 2}$ is closed under creating and deleting duplicate vertices, then $\iso{\Phi_{\s{P}}}(x) = \iso{\Phi_\s{M}}(x) + x^2$.
\end{theorem}

\begin{proof} By Proposition \ref{prop:sercomp}, the type generating series of $\Phi \circ \Omega$ is
\[ \iso{(\Phi \circ \Omega)}(x) = Z_{\Phi}(\iso{\Omega}(x), \iso{\Omega}(x^2), \iso{\Omega}(x^3), \ldots), \]
and the generating series of $\Phi \circ \ainv$ is
\[ \iso{(\Phi \circ \ainv)}(x) = Z_{\Phi}(\iso{\ainv}(x), \iso{\ainv}(x^2), \iso{\ainv}(x^3), \ldots). \]
Then, since $\iso{\Omega} = \iso{\ainv} = x - x^2$ (see proof of \cite[cor.~3.2]{2007}), we have
\begin{align*}
\iso{(\Phi \circ \Omega)}(x) &= Z_{\Phi}(x-x^2, x^2-x^4, x^3-x^6, \ldots) \\
&= \iso{(\Phi \circ \ainv)}(x).
\end{align*}

By Corollary \ref{cor:psolve}, $(\Phi_{\ge2})_{\s{P}} = \Phi_{\ge 2} \circ \Omega$. Then,
\begin{align*}
\Phi_{\s{P}} &= (\Phi_{\ge 2})_{\s{P}} + X \\
&= \Phi_{\ge 2} \circ \Omega + X \\
&= (\Phi - X) \circ \Omega + X \\
&= \Phi \circ \Omega - X \circ \Omega + X \\
&= \Phi \circ \Omega - \Omega + X,
\end{align*}
so $\Phi_{\s{P}} = \Phi \circ \Omega - \Omega + X$. Then the type generating series are also equal:
\begin{align*}
\iso{\Phi_{\s{P}}}(x) &= \iso{(\Phi \circ \Omega)}(x) - \iso{\Omega}(x) + x \\
&= \iso{(\Phi \circ \ainv)}(x) - (x - x^2) + x \\
&= \iso{(\Phi \circ \ainv)}(x) + x^2.
\end{align*}
By Corollary \ref{cor:phi}, $\Phi \circ \ainv + X^2 = \Phi_{\s{M}} + \s{E}_2$, and so
\begin{align*}
\iso{\Phi_{\s{P}}}(x) &= \iso{(\Phi \circ \ainv)}(x) + x^2 \\
&= \iso{\Phi_{\s{M}}}(x) + \iso{E_2}(x) \\
&= \iso{\Phi_{\s{M}}}(x) + x^2.
\end{align*}
Therefore, $\iso{\Phi_{\s{P}}}(x) = \iso{\Phi_{\s{M}}}(x) + x^2$. \end{proof}

Gessel and Li's result \cite[Cor.~3.2]{2007} is obtained as the special case $\Phi = \s{G}^C$. Here is another special case.
\begin{corollary} \label{cor:important} Let $r \ge 2$ be an integer, and let $_r\s{G}^C$ be the species of connected $r$-partite graphs. Then $\iso{_r\s{G}^C_{\s{P}}}(x) = \iso{_r\s{G}^C_{\s{M}}}(x) + x^2$. That is, the number of unlabeled point-determining connected $r$-partite graphs on $n$ vertices equals the number of unlabeled connected $r$-partite graphs without endpoints on $n$ vertices, for all $n \not= 2$. \end{corollary}

When we set $r = 2$, Corollary \ref{cor:important} yields an apparently new result about connected bipartite graphs. In fact, the number of unlabeled connected bipartite graphs \textit{without endpoints} is already in the OEIS \cite{oeis}. Thus, Corollary \ref{cor:important} proves that this OEIS entry also counts the unlabeled connected \textit{point-determining} bipartite graphs (except for the graph on two vertices).

\section{Acknowledgements}

The authors would like to thank the math faculty at Carleton College, in particular Rafe Jones for reviewing our paper. Most importantly, we would like to thank Andrew Gainer-Dewar for teaching us all about species theory, advising and helping us on this project, and showing us the real way to make tea.

\begin{appendices}
\section{Using Sage to do species theory} \label{appA}
This appendix serves to illustrate the concepts of species theory by using the open-source software system Sage \cite{Sage}. Sage is a language written on top of Python. It allows us to define species and use its theory to gain tangible results. With Sage, we can enumerate, calculate, and build upon different types of species.

First we can define some species that are already built-in in Sage, for example, the singleton species.

\begin{example} \end{example}

\begin{code}
sage: X = species.SingletonSpecies()
sage: print X
Singleton species
sage: x = X.structures([3])
sage: x
Structures for Singleton species with labels [3]

\end{code}

For a species $\s{F}$, we can illustrate the set of $\s{F}$-structures and the transports in Sage as in the following example:

\begin{example}
We apply the transport of permutation $(12)$ on the permutation species $\s{F}$ for the set of labels $\{1, 2, 3\}$
\end{example}

\begin{code}
sage: sigma = PermutationGroupElement((1,2))
sage: P = species.PermutationSpecies()
sage: p = P.structures([1,2,3])
sage: for i in p.list():
....:    print i,",",i.transport(sigma),"\n-------"
....:     
[1, 2, 3] , [1, 2, 3] 
-------
[1, 3, 2] , [3, 2, 1] 
-------
[2, 1, 3] , [2, 1, 3] 
-------
[2, 3, 1] , [3, 1, 2] 
-------
[3, 1, 2] , [2, 3, 1] 
-------
[3, 2, 1] , [1, 3, 2] 
-------


\end{code}

Notice that the action taking place on $[1, 3, 2]$ is $(23)$. Then by applying $(12)$ to $(23)$, we end up with $(13)$. Therefore, after the transport, we end up with the permutation $[3, 2, 1]$.

Given a species of $\s{F}$-structures, the three associated series can be computed in Sage as follows:

\begin{example}
We compute three types of associated series for the species $\s{G}$ of simple graphs:
\end{example}

\begin{code}
sage: G = species.SimpleGraphSpecies()
sage: G.generating_series()
1 + x + x^2 + 4/3*x^3 + 8/3*x^4 + O(x^5)

sage: G.isotype_generating_series()
1 + x + 2*x^2 + 4*x^3 + 11*x^4 + 34*x^5 + 156*x^6 + 1044*x^7 + 
12346*x^8 + 12346*x^9 + 12346*x^10 + 12346*x^11 + ...

sage: G.cycle_index_series()
p[]*1 + p[1]*x + (p[1,1]+p[2])*x^2 + (4/3*p[1,1,1]+2*p[2,1]+
2/3*p[3])*x^3 +
(8/3*p[1,1,1,1]+4*p[2,1,1]+2*p[2,2]+4/3*p[3,1]+p[4])*x^4 + O(x^5)

\end{code}

Sage also lets us perform the combinatorial operations on the species. 

\begin{example}
We compute the cycle index series of the species of Partitions using Sage. Notice that $\textsc{Part} = \s{E} \circ \s{E}_+$ where $\textsc{Part}$ is the species of partitions, $\s{E}$ the species of sets, and $\s{E}_+$ the non-empty species of sets. Computation in Sage tells us the following.

\begin{code}
sage: E = species.SetSpecies()
sage: Epos = species.SetSpecies(1)
sage: Part = E(Epos)
sage: Partitions = species.PartitionSpecies()

sage: Part.generating_series().coefficients(10)
[1, 1, 1, 5/6, 5/8, 13/30, 203/720, 877/5040, 23/224, 1007/17280]

sage: Partitions.generating_series().coefficients(10)
[1, 1, 1, 5/6, 5/8, 13/30, 203/720, 877/5040, 23/224, 1007/17280]

sage: Part.cycle_index_series().coefficients(6)
[p[], p[1], p[1, 1] + p[2], 5/6*p[1, 1, 1] + 3/2*p[2, 1] +
2/3*p[3], 5/8*p[1, 1, 1, 1] + 7/4*p[2, 1, 1] + 7/8*p[2, 2] +
p[3, 1] + 3/4*p[4], 13/30*p[1, 1, 1, 1, 1] + 5/3*p[2, 1, 1, 1] +
3/2*p[2, 2, 1] + 7/6*p[3, 1, 1] + 5/6*p[3, 2] + p[4, 1] + 2/5*p[5]]

sage: Partitions.cycle_index_series().coefficients(6)
[p[], p[1], p[1, 1] + p[2], 5/6*p[1, 1, 1] + 3/2*p[2, 1] +
2/3*p[3], 5/8*p[1, 1, 1, 1] + 7/4*p[2, 1, 1] + 7/8*p[2, 2] +
p[3, 1] + 3/4*p[4], 13/30*p[1, 1, 1, 1, 1] + 5/3*p[2, 1, 1, 1] +
3/2*p[2, 2, 1] + 7/6*p[3, 1, 1] + 5/6*p[3, 2] + p[4, 1] + 2/5*p[5]]

sage: Part.isotype_generating_series().coefficients(10)
[1, 1, 2, 3, 5, 7, 11, 15, 22, 30]

sage: Partitions.isotype_generating_series().coefficients(10)
[1, 1, 2, 3, 5, 7, 11, 15, 22, 30]
\end{code}

Although not a rigorous proof, the fact that we see the exact same leading coefficients for the first few terms of all three types of generating series is a good sign that we have successfully calculated the species of Partitions.
\end{example}

\begin{example}
Now, we will recursively calculate the species of Linear Orderings by illustrating the fact that $\s{L} = \mathbf{1} + X*\s{L}$ where $\s{L}$ is the species of Linear Orderings, $\mathbf{1}$ is the species of the Empty Set, and $X$ is the Singleton Species.

\begin{code}
sage: L = CombinatorialSpecies()
sage: L.define(species.EmptySetSpecies()+X*L)
sage: L.cycle_index_series().coefficients(10)
[p[], p[1], p[1, 1], p[1, 1, 1], p[1, 1, 1, 1],
p[1, 1, 1, 1, 1], p[1, 1, 1, 1, 1, 1], p[1, 1, 1, 1, 1, 1, 1],
p[1, 1, 1, 1, 1, 1, 1, 1], p[1, 1, 1, 1, 1, 1, 1, 1, 1]]
\end{code}
\end{example}

\begin{example}
Sage also allows us to set up environments that respect the nature of virtual species. Specifically, we are able to create the compositional inverse of a specific cycle index along with computing the cycle index for the combinatorial logarithm $\Omega$. We first set up the environment (the following computational work was done by Andrew Gainer-Dewar).

\begin{code}
sage: from sage.combinat.species.stream import Stream,
_integers_from
sage: from sage.combinat.species.generating_series import
CycleIndexSeriesRing
sage: CIS = CycleIndexSeriesRing(QQ)
sage: p = SFAPower(QQ)
\end{code}

Then we compute the compositional inverse of a specified cycle index.

\begin{code}
sage: def ci_compinv( f ):
....:     result = CIS()
....:     result.define(X.cycle_index_series() -
(f - X.cycle_index_series()).compose(result))
....:     return result
\end{code}

We now compute the cycle index for the combinatorial logarithm $\Omega$ (The process of calculating this cyle index is motivated by the work in \cite{LaBelle}).

\begin{code}
sage: def omegaterm(n):
....:     if n == 0:
....:         return 0
....:     elif n == 1:
....:         return p[1]
....:     else:
....:         return 1/n * ((-1)^(n-1) * p[1]**n - sum(d *
p([Integer(n/d)]).plethysm(omegaterm(d)) for d in divisors(n)
[:-1]))

sage: def omegagen():
....:     for n in _integers_from(0):
....:         yield omegaterm(n)

sage: Omega = CIS(omegagen())
\end{code}
\end{example}

\section{Using Sage to find the cycle index of \texorpdfstring{$\s{PBP}$}{PBP}} \label{appB}

This appendix illustrates the calculation of \[\s{PBP} = \s{E}\circ (\s{CBC}/S_2)\circ \Omega\] on the cycle index series level using Sage.

By Definition \ref{thm:quotspeciescis}, the cycle index \[Z_{\s{CBP}} = \frac{1}{2}(Z_{\s{CBC}}(e) + Z_{\s{CBC}}(t)).\] Thus, we can use Sage to calculate the cycle index of $\s{PBP}$.

All of this code is due to Andrew Gainer-Dewar. A preliminary version of the code appears in ~\cite{dissertation}.
\subsection*{Preliminaries}

First, we set up the Sage environment, and create the cycle index series ring and the cycle index for $\Omega$ as shown in Section 4. We also create a couple of small functions for dealing with partitions, union(mu, nu), which combines the parts of $\mu$ and the parts of $\nu$ into one larger partition, and partmult(mu, n), which returns a partition where each part is $n$ times the corresponding part of $\mu$.

\begin{code}[caption = {Set up Sage environment}, language=Python]
sage: from sage.combinat.species.stream import Stream,
_integers_from
sage: from sage.combinat.species.generating_series import
CycleIndexSeriesRing
sage: CIS = CycleIndexSeriesRing(QQ)
sage: p = SFAPower(QQ)
\end{code}

\begin{code}[caption = {Define partition functions}, language=Python]
sage: def union( mu, nu):
....:     return Partition(sorted(mu.to_list() + nu.to_list(),
            reverse=true))
    
sage: def partmult( mu, n ):
....:     return Partition([part * n for part in mu.to_list()])
\end{code}

\subsection*{Calculating $\s{BC}$ and $\s{CBC}$}

The following code calculates the cycle index for the $\Gamma$-species $\s{BC}$ for both the identity element, $e$, and the nonidentity element $t$. A description of the workings of this code is under the block.

\begin{code}[caption = {Calculate $\Gamma$--cycle index for $\s{BC}$}, language=Python]
sage: def efixedbcgraphs( mu, nu ):
....:     return 2**(sum([gcd(i, j) for i in mu for j in nu]))

sage: def ebcgen():
....:     yield p(0)
....:     for n in _integers_from(1):
....:         yield sum(p(union(pair[0], pair[1]))
                /(pair[0].aut() * pair[1].aut())
                  * efixedbcgraphs(pair[0], pair[1])
                    for pair in PartitionTuples(2, n))

sage: def tfixedbcgraphs( mu ):
....:     return 2**(len(mu) + sum([integer_ceil(p/2) for p in mu])
            + sum([gcd(mu[i], mu[j]) for i in range(0, len(mu))
              for j in range(i+1, len(mu))]))

sage: def tbcgen():
....:     yield p(0)
....:     for n in _integers_from(1):
....:         yield p(0)
....:         yield sum(tfixedbcgraphs( mu ) * p(partmult(mu, 2))
                /partmult(mu, 2).aut() for mu in Partitions(n))

sage: BC = {e: CIS(ebcgen()), t: CIS(tbcgen())}
\end{code}

Looking at this code, we see that the cycle indices $Z_{\s{CBC}}(e)$ and $Z_{\s{CBC}}(t)$ are given by the functions ebcgen() and tbcgen(). For each $n$ and the identity element of $S_2$, $\pi\in S_n$ can only fix $g\in\s{CBC}$ if $\pi$ is decomposible into a permutation $\pi_1$ of the blue vertices and a permutation $\pi_2$ of the red vertices. Therefore, ebcgen() sums over pairs of partitions, representing the cycle types of $\pi_1$ and $\pi_2$. Then the function efixedbcgraphs(pair[0], pair[1]) is called to determine how many elements of $\s{CBC}$ are fixed under $\pi$. This result is then divided by $z_\mu z\nu$, as given in \cite[Theorem~2.2.1]{2007}. efixedbcgraphs computes how many elements of $\s{CBC}$ are fixed by $\pi$ if $\pi_1$ has cycle type $\mu$ and $\pi_2$ has cycle type $\nu$, which is also given in by \cite[Theorem~2.2.1]{2007}.

The function tbcgen() works by calling tfixedbcgraphs() in a similar way, using the result of \cite[Theorem~2.2.2]{2007}.

The following code performs the $S_2$-quotient and calculates the cycle index of $\s{CBC}$ from the $\Gamma$--cycle index of $\s{BC}$. Note that the $e$-term of the cycle index is just the usual formula $Z_{\s{CBC}}(e) = Z_\Omega\circ Z_{\s{BC}}(e)$. The $t$-term of the composition is an encoding of Definitions \ref{def:gammaComposition} and \ref{def:gammaPlethysm}.

\begin{code}[caption = {Calculate $\Gamma$--cycle index for 
$\s{CBC}$}, language=Python]
sage: def CBCtermmap(term):
....:     if term == 0:
....:         return CIS(0)
....:     termbuilder = lambda part: prod(BC[t**p].stretch(p)
            for p in part)
....:     return sum(coeff*termbuilder(part) for part,coeff in term)

sage: CBC = {e: Omega.compose(BC[e]),
        t: CIS.sum_generator(CBCtermmap(Omega.coefficient(i))
          for i in _integers_from(0))}
\end{code}

\subsection*{Calculating $\s{PBP}$}

Once we have the cycle index for $\s{CBC}$, we are almost there. By Lemma \ref{lem:pbp}, we calculate the cycle index for \[\s{PBP} = \s{E}\circ (\s{CBC}/S_2)\circ \Omega\] as follows.

\begin{code}[caption = {Calculate cycle index for $\s{PBP}$}, language=Python]
sage: CBP = 1/2*(CBC[e]+CBC[t])

sage: BP = E.compose(CBP)

sage: PBP = BP.compose(Omega)

sage: PBP.coefficients(6)
[p[], p[1], 1/2*p[1, 1] + 1/2*p[2], 1/2*p[1, 1, 1] + 1/2*p[2, 1],
5/8*p[1, 1, 1, 1] + 1/4*p[2, 1, 1] + 7/8*p[2, 2] + 1/4*p[4],
9/8*p[1, 1,1, 1, 1] + 1/4*p[2, 1, 1, 1]
+ 11/8*p[2, 2, 1] + 1/4*p[4, 1]]

sage: PBP.generating_series().coefficients(10)
[1, 1, 1/2, 1/2, 5/8, 9/8, 125/48, 123/16, 11129/384, 17643/128]

sage: PBP.isotype_generating_series().coefficients(10)
[1, 1, 1, 1, 2, 3, 8, 17, 63, 224]
\end{code}

Note that the exponential generating series and the type generating series start with the term $n=0$. Thus, we have determined the cycle index for point-determining bipartite graphs. We have also done a similar calculation for the connected point-determining bipartite graphs, using Corollaries \ref{cor:pd} and \ref{cor:psolve} (not shown here). The first few terms of the labeled and unlabeled point-determining bipartite graphs and connected point-determining bipartite graphs are shown below.
\begin{table}[ht]
	\begin{tabular}	{r|l|l}
$n$ & {\bf Labeled} & {\bf Unlabeled}\\
\hline
0&1&1\\
1&1&1\\
2&1&1\\
3&3&1\\
4&15&2\\
5&135&3\\
6&1875&8\\
7&38745&17\\
8&1168545&63\\
9&50017905&224\\
10&3029330745&1248\\
11&257116925835&8218\\
12&30546104308335&75992\\
13&5065906139629335&906635\\
14&1172940061645387035&14447433\\
15&379092680506164049425&303100595\\
16&171204492289446788997825&8415834690\\
17&108139946568584292606269025&309390830222\\
18&95671942593719946611454522225&15105805368214\\
19&118699636146295502809945048489875&982300491033887\\
20&206821794864679268333769991824317775&85356503319933261
	\end{tabular}
    \caption{Enumeration of point-determining bipartite graphs on $n$ vertices.}
\end{table}

\begin{table}[ht]
	\begin{tabular}	{r|l|l}
$n$ & {\bf Labeled} & {\bf Unlabeled} \\
\hline
0 & 0 & 0 \\
1 & 1 & 1\\
2 & 1 & 1\\
3 & 0 & 0\\
4 & 12 & 1\\
5 & 60 & 1\\
6 & 1320 & 5\\
7 & 26880 & 9\\
8 & 898800 & 45\\
9 & 40446000 &160\\
10 & 2568736800 &1018\\
11 & 225962684640 & 6956\\
12 & 27627178692960 & 67704\\
13 & 4686229692144000 & 830392\\
14 & 1104514965434200320 & 13539344\\
15 & 361988888631722352000 & 288643968\\
16 & 165271302775469812521600 & 8112651795\\
17 & 105278651889065640047462400 & 300974046019\\
18 & 93750696652129931568573619200 &14796399706863 \\
19 & 116899866711712459270623087360000 &967194378235406\\
20 & 204465611975190360222598610427187200 &84374194347669628
	\end{tabular}
    \caption{Enumeration of connected point-determining bipartite graphs on $n$ vertices.}
\end{table}

\end{appendices}

\pagebreak
\begin{figure}[h!]
\[\begin{tikzpicture}

\node[draw=none] at (0,0){};
\end{tikzpicture}\]
\end{figure}

\pagebreak

\begin{figure}[h!]
\[\begin{tikzpicture}

\node[draw=none] at (0,0){};
\end{tikzpicture}\]
\end{figure}

\pagebreak

\end{document}